\documentclass[11pt]{article}
\usepackage[T1]{fontenc}
\usepackage[utf8]{inputenc}
\usepackage{epsfig}
\usepackage[margin=1in]{geometry}
\usepackage{amssymb, amsmath, mathrsfs, dsfont, amsthm, graphicx, hyperref, blindtext,todonotes,soul, caption, enumitem, color}
\usepackage{subfigure}
\usepackage{mathtools}
\usepackage{hyperref}
\usepackage{tabularx}
\usepackage{multirow}
\usepackage{booktabs}
\usepackage{authblk}
\mathtoolsset{showonlyrefs}
\usepackage{xcolor}
\usepackage{colortbl}

\mathtoolsset{showonlyrefs}

\newtheorem{lemma}{Lemma}[section]
\newtheorem{theorem}[lemma]{Theorem}
\newtheorem{corollary}[lemma]{Corollary}

\newtheorem{proposition}[lemma]{Proposition}
\newtheorem{definition}[lemma]{Definition}

\newtheorem{remark}[lemma]{Remark}

\newtheorem{notation}[lemma]{Notation}

\def\be{\begin{align}}
\def\ee{\end{align}}
\def\bea{\begin{align}}
\def\eea{\end{align}}
\def\bpm{\begin{pmatrix}}
\def\epm{\end{pmatrix}}

\def\nn{ }
\def\lb{\label}
\def\bs{\setminus}
\def\pt{\partial}

\def\R{{\bf R}}

\def\A{{\bf A}}

\def\aa{{\alpha}}
\def\bb{{\beta}}
\def\ga{{\gamma}}

\def\ep{{\epsilon}}

\def\lm{{\lambda}}
\def\Lm{{\Lambda}}
\def\dl{{\delta}}
\def\Dl{{\Delta}}
\def\sg{{\sigma}}

\def\Sg{{\Sigma}}

\def\<{{\langle}}
\def\>{{\rangle}}
\def\d{{\mathrm{d}}}
\def\pt{\partial}

\def\cR{{\cal R}}

\def\cC{{\cal C}}

\def\P{{\cal P}}

\def\Nn{{\cal N}}
\def\O{{\cal O}}

\def\diag{{\rm diag}}

\def\r{\right}
\def\l{\left}
\def\ii{\sqrt{-1}}
\def\td{\tilde}
\def\sgn{\mathrm{sgn}}
\def\RRe{{\rm Re}}
\def\IIm{{\rm Im}}

\def\td#1{\tilde{#1}}

\title{Mathematical Mechanism on Dynamical System Algorithms of the Ising Model}
\author[a]{Bowen LIU \thanks{Email: bowen.liu@sjtu.edu.cn. Partially supported by NSFC (No. 11671215) and Innovation Program of Shanghai Municipal Education Commission.}}
\author[a]{Kaizhi WANG \thanks{Email: kzwang@sjtu.edu.cn. Partially supported by Natural Scientific Foundation of China (Grant No. 11771283 and No.
		11931016) and Innovation Program of Shanghai Municipal Education Commission.}}
\author[a]{Dongmei XIAO \thanks{Corresponding author. Email: xiaodm@sjtu.edu.cn. Partially supported by Innovation Program of Shanghai Municipal Education Commission.}}
\author[b]{Zhan YU \thanks{Email: zhanyu2-c@my.cityu.edu.hk}}
\affil[a]{School of Mathematical Sciences, Shanghai Jiao Tong University, Shanghai 200240, China}
\affil[b]{Department of Mathematics, City University of Hong Kong, Kowloon, Hong Kong}

\date{}
\allowdisplaybreaks
\begin{document}

\maketitle

\begin{abstract}
	Various combinatorial optimization NP-hard problems can be reduced to finding the minimizer of an Ising model, which is a discrete mathematical model. It is an intellectual challenge to develop some mathematical tools or algorithms for solving the Ising model. Over the past decades, some continuous approaches or algorithms have been proposed from physical, mathematical or computational  views for optimizing the Ising model such as quantum annealing, the coherent Ising machine, simulated annealing, adiabatic Hamiltonian systems, etc..  However, the mathematical principle of these algorithms is far from being understood. In this paper, we reveal the mathematical mechanism of dynamical system algorithms for the Ising model by Morse theory and variational methods. We prove that the dynamical system algorithms can be designed to minimize a continuous function whose local minimum points give all the candidates of the Ising model and the global minimum gives the minimizer of Ising problem. Using this mathematical mechanism, we can easily understand several dynamical system algorithms of the Ising model such as the coherent Ising machine, the Kerr-nonlinear parametric oscillators and the simulated bifurcation algorithm.
Furthermore, motivated by the works of C. Conley, we study transit and capture properties of the simulated bifurcation algorithm to explain its convergence by the low energy transit and capture in celestial mechanics. A detailed discussion on $2$-spin and $3$-spin Ising models is presented as application.
\end{abstract}



\renewcommand{\theequation}{\thesection.\arabic{equation}}
\renewcommand{\thefigure}{\thesection.\arabic{figure}}

\setcounter{equation}{0}
\setcounter{figure}{0}
\section{Introduction and main results}%{Section 1}
\label{sec:1}

The Ising model or Lenz-Ising model has been widely studied in combinatorial optimization and statistical physics. This model was first proposed by W. Lenz, and in 1925 its one-dimensional case was solved by E. Ising in his thesis \cite{ising1925beitrag}. From the statistical physical point of view, the Ising model is regarded as a translation-invariant, ferromagnetic spin system. To study this spin system, many elegant and profound theories in probability and statistical physics have been developed in recent years. Reader may refer to \cite{Aizenman2014, Aizenman2019, DuminilCopin2019, DuminilCopin2018} and references therein for more details. In the aspect of combinatorial optimization, many NP-hard problems, for example, max-cut problem, can be equivalently formulated as Ising models (cf. \cite{barahona1988application,inagaki2016coherent,McMahon2016}).

As a combinatorial optimization problem, we consider the Ising model without an external field as follows.
\begin{align}
	\min_v \;\;  E(v) := - \frac{1}{2} v^T S v,  \lb{eqn:ising}
\end{align}
where $v=(v_1,\cdots,v_n) \in \{1,-1\}^n$ is the candidate and $S=(s_{ij})_{n\times n}$ is the symmetric coupling coefficient matrix with $s_{ii} = 0$ for all $i$. Each $v_i$ denotes the $i$th Ising spin, $v$ is the vector representation of a spin configuration, and $v^T$ denotes the transpose of $v$. We call $E(v)$ {\it the Ising energy} and denote {\it the candidate set} of the Ising model by $C(E) = \{-1,1\}^n$.

According to \cite{barahona1982computational}, minimizing the Ising energy $E(v)$ in \eqref{eqn:ising} belongs to the class of the non-deterministic polynomial-time (NP)-hard problem. It is an important topic in combinatorial optimization (cf. \cite{barahona1988application, Dobrushin1973,MR0438971,Galashin2020,MR3627847,Jerrum1993,Nishimori2001, Ott2019,MR3630933, Stein1992,  Yoshimura2019}).
Over the past decades, physicists and computer scientists tried to find a proper model or algorithms to solve the Ising problem \eqref{eqn:ising}.
The quantum annealing was used to study the ground state of the Ising problem (cf. \cite{Johnson2011}, \cite{Kim2010} and \cite{Santoro2002}). The electromechanical system can be also applied to minimize the Ising problem in \cite{Mahboob2016}, etc.. In this paper, we focus on the two dynamical system algorithms: the coherent Ising machines (CIM) and the adiabatic Hamiltonian systems. We refer readers to \cite{Sahai2020Dynamical} for other continuous models and algorithms on combinatorial optimization.

In 2011, Utsunomiya et al. proposed one Ising machine based on optical coherent feedback in \cite{utsunomiya2011mapping}. Since 2013, the coherent Ising machine was proposed to solve the Ising problem by the similarity
between the Ising problem and the Hamiltonian of bistable interfering coherent optical states (cf. \cite{bohm2019poor, haribara2016computational, marandi2014network, takata201616, wang2013coherent, yamamoto2017coherent}).
Especially, in 2016, Inagaki et al. in \cite{inagaki2016coherent} applied CIM to study 2000-node of Ising problems, which were outperformed simulated annealing in \cite{Kirkpatrick1983}.
On the other hand, based on the adiabatic Hamiltonian systems and quantum adiabatic optimization,
the Kerr-nonlinear
parametric oscillators (KPO) in \cite{goto2019quantum} and the simulated bifurcation (SB) algorithm in \cite{goto2019combinatorial} were introduced to minimize of Ising model by classical computers in 2019.

According to some experiments (cf. Figure 2 of \cite{goto2019combinatorial}), it is shown that the CIM and the adiabatic Hamiltonian systems perform better than the simulated annealing. However, a natural question arises if
the global minimum points given by above dynamical system algorithms correspond to the minimizers of Ising problem. In this paper, we will answer this question and prove that  minimizing Ising model \eqref{eqn:ising} can be realized by minimizing the following smooth function.

Define a function $U: \R^n \to \R$ by
\begin{align}
	U(x) = \sum_{i=1}^{n}\frac{1}{4} x_{i}^{4}+\frac{\bb - \aa^2}{2}x^Tx  - \frac{1}{2} x^T S x, \lb{eqn:pote}
\end{align}
where $x=(x_1, x_2, \dots, x_n) \in \R^n$, $\aa$ is a positive parameter, $\bb$ is a given positive constant, and $S$ is the given matrix in \eqref{eqn:ising}.
Via Morse theory, we prove
that minimizing Ising model \eqref{eqn:ising} can be realized by minimizing the smooth function $U(x)$ in \eqref{eqn:pote}. Hence, there exists a correspondence between the global minimum points of  the Ising model and those of the smooth function $U(x)$.
This correspondence can be applied to understand the mechanism of CIM models in \cite{wang2013coherent,  yamamoto2017coherent}, the KPO in \cite{goto2019quantum} and SB algorithm \cite{goto2019combinatorial} mathematically.

Let the signum vector of $x$ be
\begin{align}
	\sgn(x) := (\mathrm{sgn}(x_1),\cdots,\mathrm{sgn}(x_n)) \in \{-1,0,1\}^n:=\overbrace{\{-1,0,1\} \times \{-1,0,1\} \times \cdots \times \{-1,0,1\}}^{n}.
\end{align}
We first prove that there exists a constant $\aa_0> 0$ such that the signum vectors of the local minimum points of $U(x)$ are $C(E)=\{1,-1\}^n$ for any $\aa > \aa_0$ (cf. Proposition \ref{prop:3n.cp} below). Then we obtain the main result as follows.

\begin{theorem}\lb{thm:main}
	For any given $\bb$ and $S$, there exists $\aa_* \geq  \aa_0$ such that for $\aa > \aa_*$, if $x_0$ is a global minimum point of $U(x)$, then the signum vector of $x_0$ is a minimizer of the Ising problem.
\end{theorem}

Using Theorem \ref{thm:main}, we study  the Ising problem in $\R^2$ and an Ising problem in $\R^3$.
In $\R^2$, any Ising model can be reduced to $E = -\frac{1}{2}v^T Sv$,
where $v\in \{-1,1\}^2$ and $S = S_2 = (\begin{smallmatrix}
	0 & 1 \\1 & 0
\end{smallmatrix})$.
The number of critical points of $U(x)$ depends on $\aa$. To show all bifurcations of critical points as $\aa$ increases, we assume that $\bb > 1$ is fixed. In this case, $\aa_*= \aa_0 = \sqrt{\bb+2}$. When $\aa > \sqrt{\bb+2}$ in Theorem \ref{thm:main}, $U(x)$ possesses $4$ local minimum points. The signum vectors of the global minimum points $(\lm_1, \lm_1)$ and $(-\lm_1, -\lm_1)$ are $(1,1)$ and $(-1,-1)$ respectively. Both $(1,1)$ and $(-1,-1)$ are the minimizers of this Ising problem. The change of local maximum, saddle and local minimum are given in Table \ref{tab:1} and shown in Figure \ref{fig:contourplot}, where those $\lm_i> 0$ will be given in \eqref{eqn:r2.lm1}-\eqref{eqn:r2.lm4}. In $\R^3$, we arbitrarily give a matrix $S_3$, for example $S_3 = \l(\begin{smallmatrix}
	0 & 1 & -2\\
	1 & 0 & 3 \\
	-2 & 3 & 0
\end{smallmatrix}\r)$. We apply Theorem \ref{thm:main} to study Ising model in \eqref{eqn:ising} with $S=S_3$ in Section \ref{sec:opt.app}. The minimizers of the Ising energy $E$ are $(-1,1,1)$ or $(1, -1,-1)$. Numerical computations show that when $\aa > 4.6$, the sigum vectors of global minimum points of $U(x)$ are $(-1,1,1)$ or $(1, -1,-1)$.  More details on bifurcations of $U(x)$ with $S = S_2$ and $S = S_3$, respectively,  will be discussed in Section \ref{sec:opt.app}.

\begin{table}[htbp]
	\centering
	\caption{{\footnotesize The critical points of $U(x)$ when $n =2$ are given here. If $\aa > \aa_*\equiv \sqrt{\bb+2}$, both $(\lm_1,\lm_1)$ and $(-\lm_1,-\lm_1)$ are the global minimum points and their signum vectors $(-1,-1)$ and $(1,1)$ minimize the Ising energy $E=-\frac{1}{2}v^T S_2 v$.}}
	\begin{tabular}{llll}	
		\toprule
		$\aa$ &  Min & Saddle & Max\\
		\midrule
		\rowcolor[gray]{.9}
		$\aa^2 < \bb-1$ 		     & $(0,0)$                             & NA & NA\\
		$\aa^2 \in (\bb-1, \bb+1)$   & $(\lm_1,\lm_1)$, $(-\lm_1,-\lm_1)$  & $(0,0)$ & NA\\
		\rowcolor[gray]{.9}
		$\aa^2\in (\bb +1, \bb + 2)$ & $(\lm_1,\lm_1)$, $(-\lm_1,-\lm_1)$  & $(\lm_2, -\lm_2)$, $(-\lm_2,\lm_2)$ & $(0,0)$ \\
		\multirow{2}{*}{$\aa^2 > \bb +2$} & $(\lm_1,\lm_1)$, $(-\lm_1,-\lm_1)$   & $(\lm_3, -\lm_4)$, $(-\lm_3, \lm_4)$ & \multirow{2}{*}{$(0,0)$}\\
		& $(\lm_2, -\lm_2)$, $(-\lm_2,\lm_2)$ & $(\lm_4, -\lm_3)$, $(-\lm_4, \lm_3)$ &\\
		\bottomrule\\	
	\end{tabular}\lb{tab:1}
\end{table}

\begin{figure}
	\centering
	\begin{tabular}{cccc}
		\includegraphics[width=0.21\textwidth]{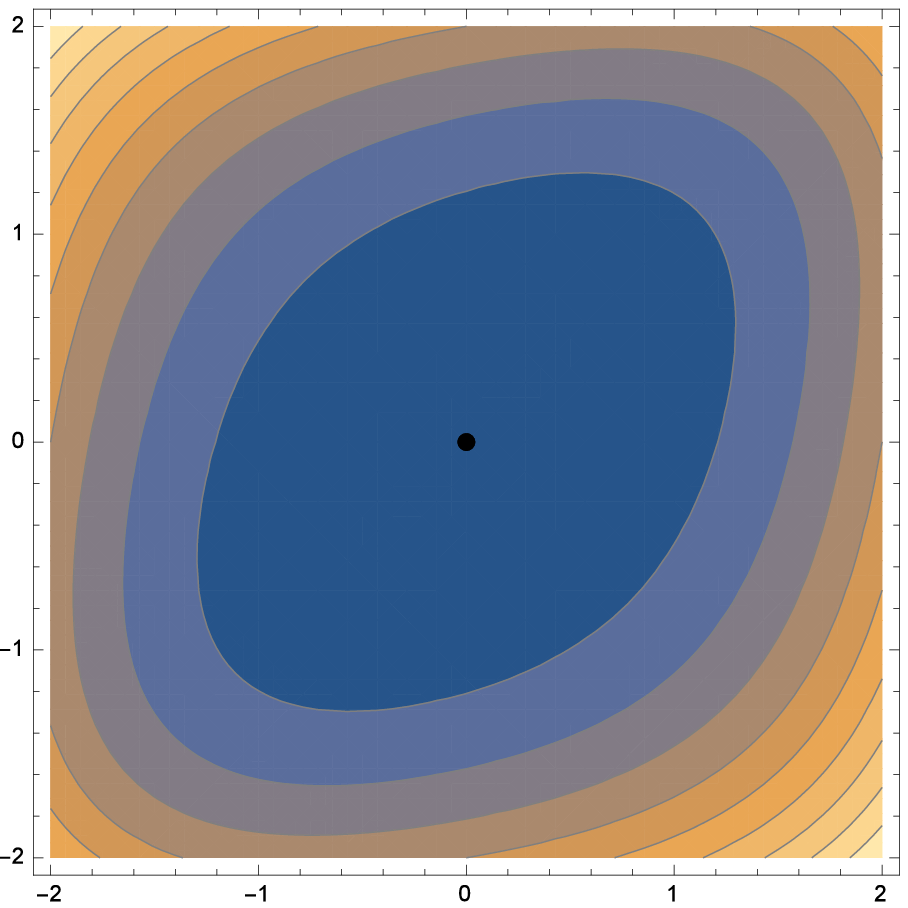} &
		\includegraphics[width=0.21\textwidth]{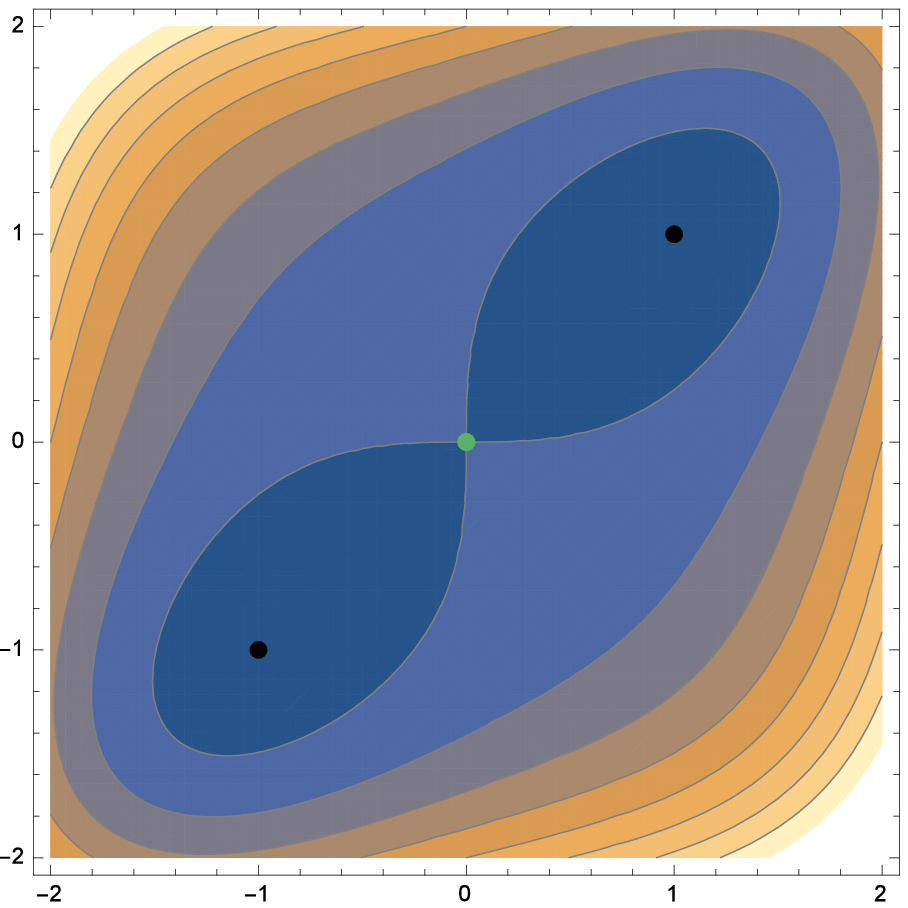}  & \includegraphics[width=0.21\textwidth]{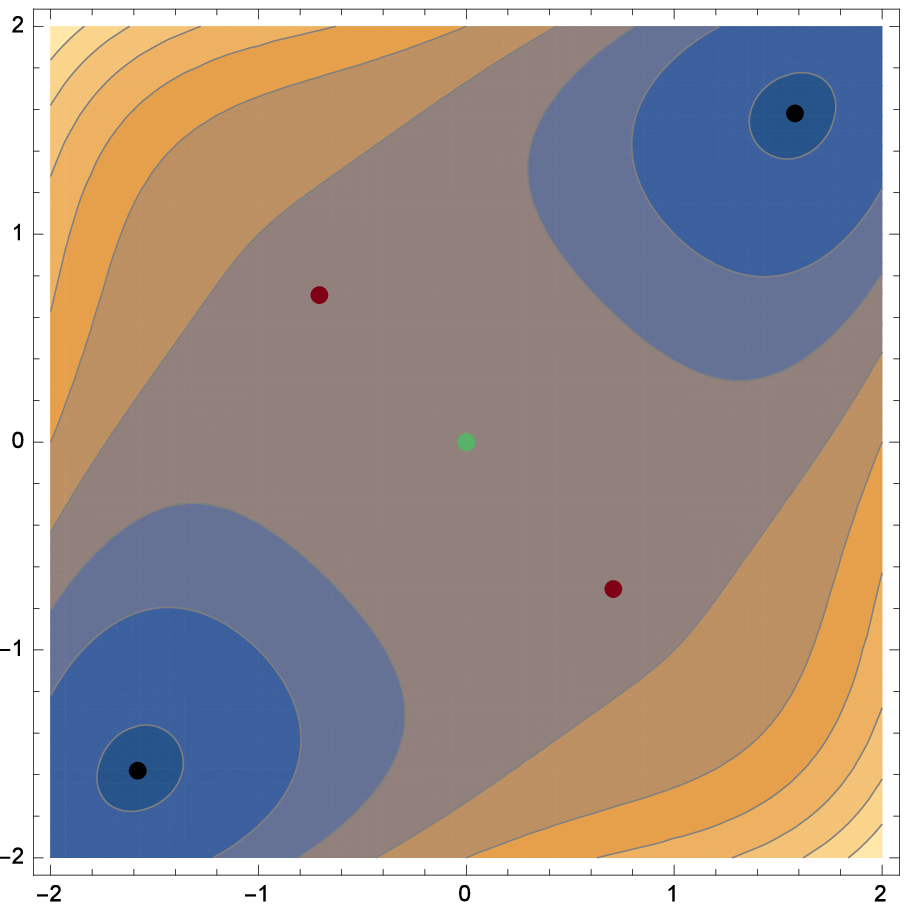} &
		\includegraphics[width=0.21\textwidth]{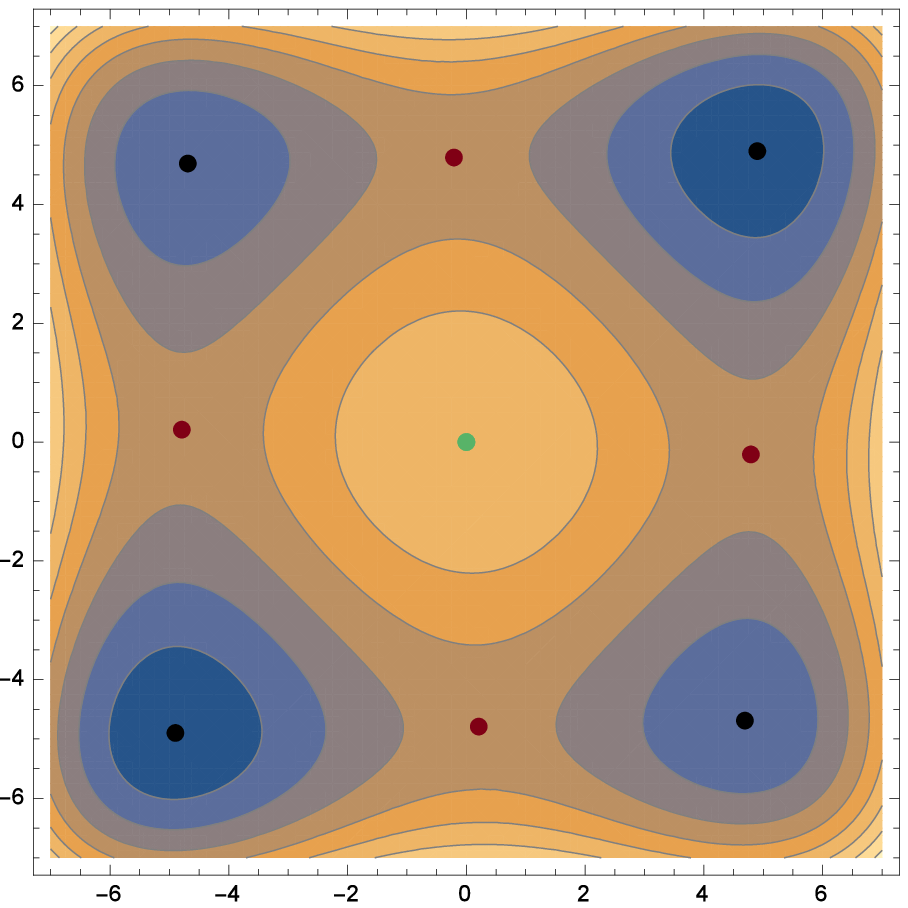}\\
		(a)  & (b) & (c) & (d)   \\[4pt]
	\end{tabular}
	\caption{\footnotesize{In $\R^2$, the contour plots of $U(x) = \frac{1}{4} (x_1^4 + x_2^4) + \frac{2-\aa^2}{2} (x_1^2 + x_2^2) - x_1x_2$ depend on $\aa$ where $\bb = 2$. When $\aa \geq \aa_* = 2$, Theorem \ref{thm:main} holds. The black dots are local minimum points; the red dots are saddles; and the green dots are local maximum points. The case of $\aa = 0$ is shown in  (a) where $(0,0)$ is the only local minimum point. The case of $\aa = \sqrt{2}$ is shown in (b) where there are only one saddle $(0,0)$ and two local minimum points $(\lm_1,\lm_1)$ and $(-\lm_1,-\lm_1)$. The case of $\aa =\sqrt{7/2}$ is shown in (c) where there are only two saddles $(\lm_2, -\lm_2)$ and $(-\lm_2,\lm_2)$, two local minimum points $(\lm_1,\lm_1)$ and $(-\lm_1,-\lm_1)$, and a unique local maximum point $(0,0)$. The case of $\aa = 5$ is shown in (d), where there are four saddles: $(\lm_3, -\lm_4)$, $(-\lm_3, \lm_4)$, $(\lm_4, -\lm_3)$ and $(-\lm_4, \lm_3)$, four local minimum points $(\lm_1,\lm_1)$, $(-\lm_1,-\lm_1)$, $(\lm_2, -\lm_2)$ and $(-\lm_2,\lm_2)$ and a unique local maximum point $(0,0)$. When $\aa^2 > 1$,  $(\lm_1,\lm_1)$ and $(-\lm_1,-\lm_1)$ are the global minimum points shown in (b), (c) and (d). }}
	\label{fig:contourplot}
\end{figure}

Therefore, Theorem \ref{thm:main} can be applied to reveal the mathematical mechanism of the CIM and adiabatic Hamiltonian systems. We prove that the global minimum points, which are found by the CIM in \cite{wang2013coherent,  yamamoto2017coherent}, the KPO in \cite{goto2019quantum} and the SB-algorithm in \cite{goto2019combinatorial}, are the minimizer of Ising model. These results are given in Proposition \ref{prop:cim}-\ref{prop:cim2} and Proposition \ref{prop:mini.KPO.SB}-\ref{prop:mini.KPO.SB2}. Mathematically, CIM is formulated by gradient descent flows; while the KPO and SB algorithm are formulated by adiabatic Hamiltonian systems. Especially, the SB algorithm is based on a mechanical Hamiltonian system whose Hamiltonian function is the sum of kinetic energy and the potential. Even though these algorithms are based on different physical models and different dynamical systems, we can reveal their mechanism by Theorem \ref{thm:main}.

We further study the transition and convergence of SB algorithm by some tools in celestial mechanics.
In the study of restricted three body problem (e.g., the earth-moon-satellite system), the transit is used to describe orbit of zero-mass body moving from one primary to another primary through the saddle Lagrangian point between two primaries. Inspired by the low energy transfer of C. Conley in \cite{MR233535} and ballistic capture of E. Belbruno in \cite{conley1969ultimate} and \cite{MR2029316}, we employ the concepts of transit and capture from   {\it celestial mechanics} to study the transition and convergence of SB algorithm in \eqref{eqn:SB.11}. We use the concepts to mimic the motion of orbit $x(t)$ from the neighborhood of one local minimum point to the neighborhood of another local minimum point via the saddle between them. The capture in celestial mechanics is used to describe that the motion of a satellite will be in some neighborhood of one primary forever. Namely it is captured by this primary. Thus,  the concept of capture is used to describe that the orbit $x(t)$ is in some neighborhood of the local minimum point forever.

Via re-scaling, we rewrite the Hamiltonian system in SB algorithm as
\begin{equation}\lb{eqn:SB.11}
	\begin{dcases}
		\dot{x}_{i}=y_{i}, \\
		\dot{y}_{i}=-\left(x_{i}^{2} + \bb -\aa^2(t)\right) x_{i}+ \sum_{j=1}^{n} s_{ij} x_{j},\\
		i = 1, \dots, n.
	\end{dcases}
\end{equation}
The corresponding Hamiltonian $H (x,y,t)$ is given by
\begin{align}
	H (x,y,t) =\sum_{i=1}^{n} \frac{1}{2}y_{i}^{2} + U(x,t) =\sum_{i=1}^{n} \frac{1}{2}\dot{x}_{i}^{2} + \sum_{i=1}^{n} \l( \frac{1}{4}x_{i}^{4}+\frac{\bb-\aa^2(t)}{2}x_{i}^{2}\r) -\frac{1}{2} x^T Sx, \lb{eqn:SB.Hami}
\end{align}
where
\begin{align}
	U(x,t) =  \sum_{i=1}^{n} \l( \frac{1}{4}x_{i}^{4}+\frac{\bb-\aa^2(t)}{2}x_{i}^{2}\r) -\frac{1}{2} x^T Sx,
\end{align}
$\aa(t)\in C^{1}([0,\infty),\R)$ and $\bb>0$ is a given constant.

First we consider the case $\dot \aa(t) =0$. Namely $\aa(t)\equiv \aa$ and the Hamiltonian system \eqref{eqn:SB.11} is autonomous.
The  component $x(t)$ of the solution $(x(t),y(t))$ is called an {\it orbit} which is an analogue of the orbit of a star or a satellite in celestial mechanics.  For the given Hamiltonian energy $H(x,y) = c$, we define the Hill's region
\begin{align}
	\cR_c = \{x|U(x) < c\}, \nn
\end{align}
which is one classical concept of sub-level set of the potential $U(x)$ in celestial mechanics (cf. Section 5.5 of \cite{Frauenfelder2018Book}).
Since $\frac{\d H(x,y)}{\d t} = 0$ along the solution $(x(t),y(t))$ of \eqref{eqn:SB.11}, the Hamiltonian energy of solution is preserved.

In the Hill's region $\cR_c$, an orbit $x(t)$ is {\it transit} on $I\subset \R$ if there exist $t_1$ and $t_2 \in I$ such that $x(t_1)$ is in some neighborhood of a local minimum $x_1$ while $x(t_2)$ is in some neighborhood of another local minimum $x_2$. It is {\it capture} if there exists $t_3$ such that $x(t)$ can not be in a neighborhood of the others for $t > t_3$. The precise definition will be given in Definition \ref{def:cap.tst} below.

Let $\aa$ satisfy that $\aa > \aa^*$ where $\aa^*$ is as in Theorem \ref{thm:main}.
We define $U_s = \min_{x\in \cC_s(U)} U(x)$ where $\cC_s(U)$ is the set of the saddles of $U(x)$. Exploring the topology of the Hill's region and applying mountain pass theorem in variational method, we have the following theorem.
\begin{figure}
	\centering
	\begin{tabular}{cccc}
		\includegraphics[width=0.21\textwidth]{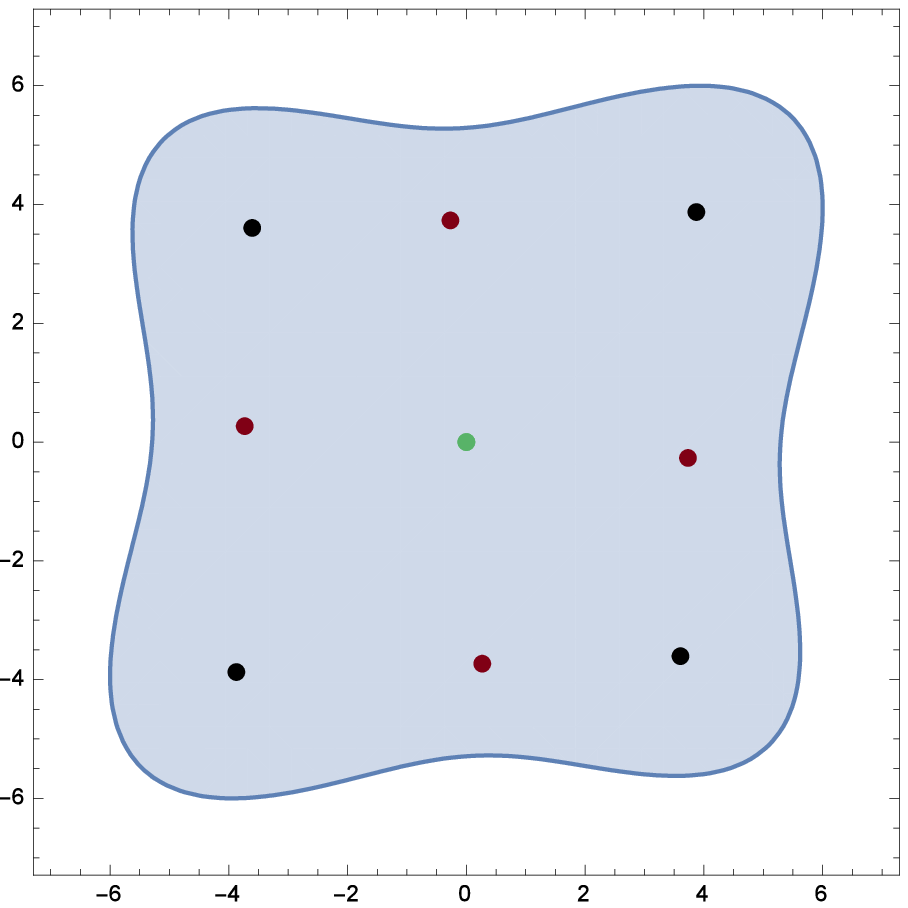} &
		\includegraphics[width=0.21\textwidth]{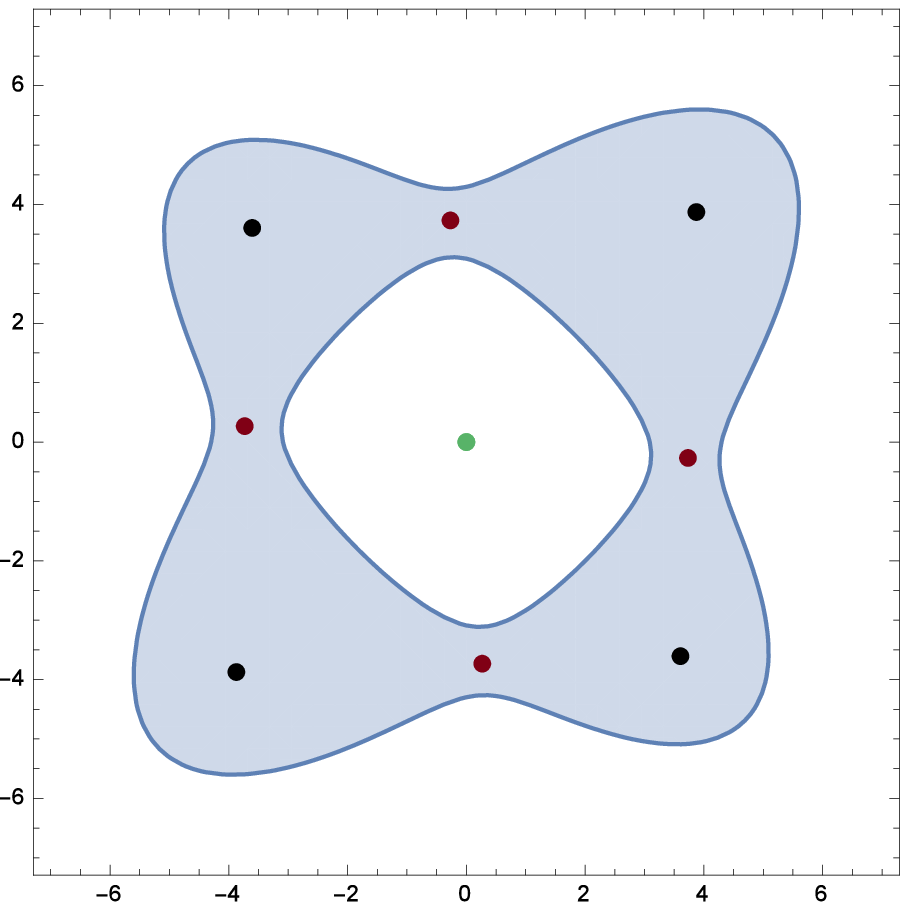} &
		\includegraphics[width=0.21\textwidth]{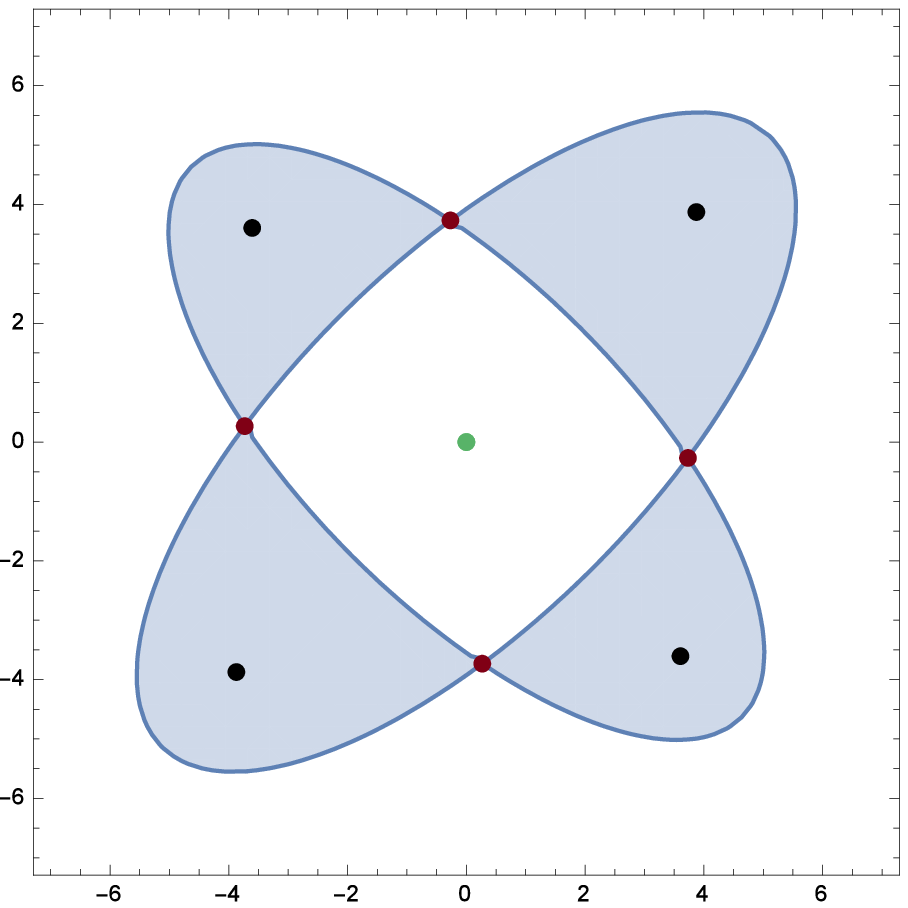} &
		\includegraphics[width=0.21\textwidth]{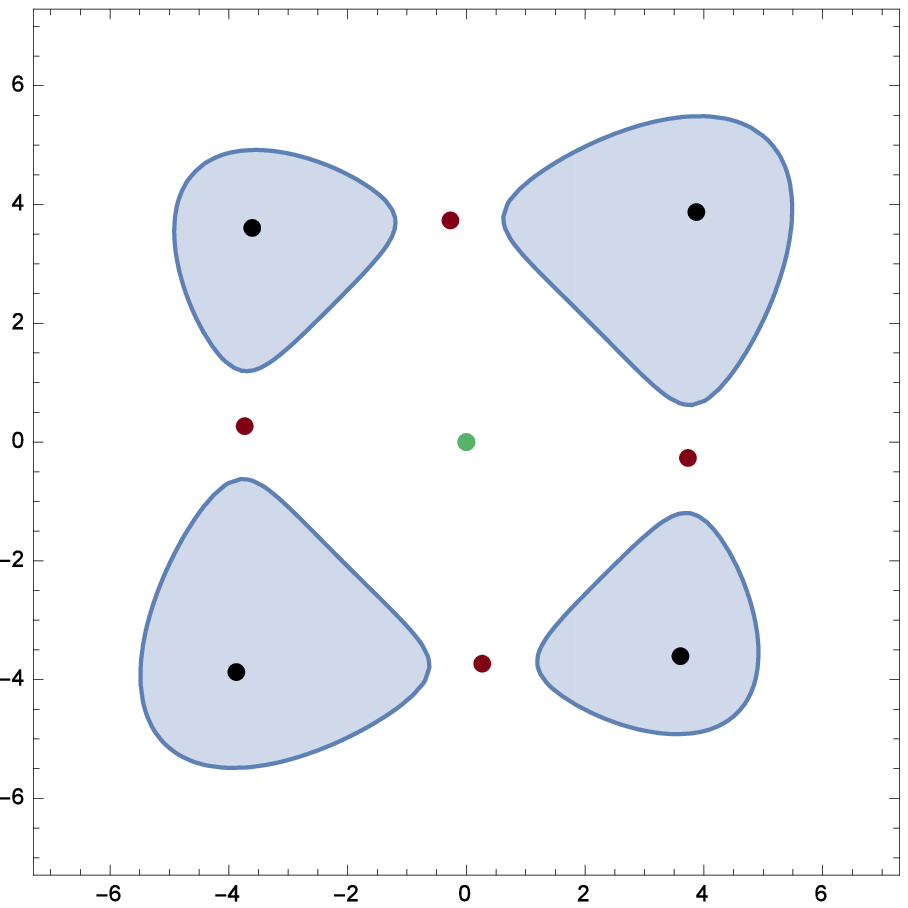}		\\
		(a)  & (b) & (c) & (d) \\[4pt]
	\end{tabular}
	\caption{ \footnotesize{		
			We take $U(x) = \frac{1}{4} (x_1^4 + x_2^4) + \frac{2-\aa^2}{2} (x_1^2 + x_2^2) - x_1x_2$ where $\aa = 4$ as an example. In these four figures, we use the blue regions to denote the Hill's regions. The green dots are the local maxima; the red dots are the saddle; and the black dots are the local minimum points. Both $(\sqrt{15},\sqrt{15})$ and $(-\sqrt{15},-\sqrt{15})$ are the global minimum points.
			When $c \geq 0$, $\cR_c$ is one simply connected closed set shown in (a) and all the critical points are contained in $\cR_c$.
			When $c \in [c_1, 0)$ with $c_1 = -\frac{98}{2}$, $\cR_c$ is still connected but not simple connected shown in (b) and there are saddles and local minimum points are contained in $\cR_c$. When $c \gtrsim c_1$, the ``necks'' are shown as the red dots in (c). When $c \in [c_2, c_1)$, $\cR_c$ is the union of four simply connected components shown in (d) and only local minimum points are contained in $\cR_c$. When $c\in [c_1,0)$,  the transit may happen which is shown in (a)-(c); when $c < c_1$, the transit cannot happen which is shown in (d).}}
	\label{fig:Hill.R2}
\end{figure}

\begin{theorem}\lb{thm:mpt}
	If $x(t)$ is a transit orbit in Hill's region $\cR_c$ with the Hamiltonian energy $c$, then $c\geq U_s$; if $c < U_s$, then $x(t)$ is a capture orbit in $\cR_c$.
\end{theorem}

By Theorem \ref{thm:mpt}, for any orbit $x(t)$, if its Hamiltonian energy is lower than  $U_s$, then the signum vector of $x(t)$ is a constant vector.

To apply Theorem \ref{thm:mpt} in $\R^2$, we further study dynamics at the saddles which are called the ``neck'' in Section \ref{sec:transit.R2} by the low energy transit orbit in \cite{MR233535}.  We find three types of orbits at the ``neck'' which are asymptotic orbits, saddle transit orbits and saddle non-transit orbits.

When $\dot \aa(t) =0$, the Hamiltonian of \eqref{eqn:SB.Hami} is conserved along any solutions. It is impossible to achieve the global minimum of $H(x,y)$ along any orbits of solutions. Hence, we consider the case that $\dot \aa(t) > 0$ as in \cite{goto2019combinatorial} where the Hamiltonian $H(x,y,t)$  decreases along solutions by $\frac{\d H }{\d t} < 0$.
However, the lowest saddle potential energy $U_s$ also decreases with $t$ in this case. Thus we need to define
a capture set which depends on $t$. Before that, we assume that $\dot \aa(t) >0$ and $\lim_{t\to \infty} \aa(t) = \aa_{\infty}> 8\aa_*$ where $\aa_*$ is as in Theorem \ref{thm:main}. Then there exists $t_0$ such that $\aa(t) > \frac{1}{4} \aa_{\infty}$ for all $t > t_0$.
Therefore, when $t > t_0$, the correspondence between the global minimum point of $U(x)$ and the minimizer of $E$ in Theorem \ref{thm:main} holds because $\aa(t) > 2 \aa_*$.

Define {\it  the capture set} $\P(t)$ of system \eqref{eqn:SB.11} as
\begin{align}\lb{eqn:p.t}
	\P (t) :=  \bigg\{x(t)\in \R^n \bigg| &|x(t)|^2 > R_0; \; H(x,\dot{x}, t) \leq \min\{U_{R_0}(t), U_B(t)\}, \\
	& \;  (x(t), \dot{x}(t)) \; \text{is a solution of \eqref{eqn:SB.11}} \bigg\}, \nn
\end{align}
where $R_0$ is a constant defined by \eqref{eqn:R0}, $U_{R_0}(t)$ and $U_B(t)$ are two functions of $t$ defined in \eqref{eqn:UR0} and \eqref{eqn:U_B} respectively.

\begin{theorem}\lb{thm:cap}
	Suppose $x(t) \in C^2([0,\infty), \R^n)$ is an orbit of the system \eqref{eqn:SB.11}, $\dot \aa(t)>0$ and $\lim_{t\to \infty} \aa(t) =\aa_{\infty}$. If there exists  $t_*> t_0$ with $x(t_*) \in \P(t_*)$, then $x(t) \in \P(t)$ for all $t\geq t_*$.
\end{theorem}

By Theorem \ref{thm:mpt} and Theorem \ref{thm:cap}, we further prove that $x(t)$ is captured in $\P (t)$ and $\sgn(x(t))$ will be fixed for all $t \geq t_*$.
If $x(t)$ is in the capture set at some $t_*$, then SB Algorithm on solving \eqref{eqn:SB.11} can be {\it stopped} because $\sgn (x(t))$ is fixed for $t \geq t_*$.
This can be interpreted as one analogue of convergence.
As an example, we will discuss this convergence for the Ising problem in $\R^2$  in Section \ref{sec:cap.r2}.

Summarizing, we reveal the correspondence that the minimizer of the Ising problem corresponds exactly to global minimum points of the dynamical system algorithms, and provides a rigorous theoretical mathematical foundation for these algorithms.
Furthermore, based on our results, some eminent Ising model-related dynamical systems, including CIM and adiabatic Hamiltonian systems, can be explained.
Moreover, we introduce the capture and transit to describe the convergence behavior of SB algorithm, which provide a novel aspect to understand this algorithm.
To the best of our knowledge, this is the first work to give the mathematical mechanism of dynamical system algorithms for the Ising model and discuss the convergence of algorithms by using celestial mechanics, Morse theory and variational methods.

This paper is organized as follows. In Section \ref{sec:opt.set}, we introduce necessary preliminaries on Morse theory, prove the main result Theorem \ref{thm:main} and give two examples to explain our results. In Section \ref{sec:app.mds}, we revisit some dynamical system algorithms (CIM,  KPO and SB algorithm) and  prove the global minimum points found by CIM,  KPO and SB algorithm are minimizers of the Ising model mathematically. In Section \ref{sec:tran.cap}, we discuss the transit and capture of SB algorithm, and prove Theorem \ref{thm:mpt}-\ref{thm:cap}.

\setcounter{equation}{0}
\setcounter{figure}{0}
\section{Mathematical mechanism of continuous models on the Ising model}
\lb{sec:opt.set}
In this section, we first introduce some preliminaries on Morse theory, then transfer the minimizing Ising model in combinatorial optimization to looking for the global minimum points of a smooth function and prove  our main result Theorem \ref{thm:main}. Last  we give two examples in $\R^2$ and $\R^3$ to explain our results.

\subsection{Mathematical analysis on the continuous models}

The {\it Morse index} of a function $f$ at a critical point is given as follows.

\begin{definition}
	Suppose that $f$ is a smooth real value function on $\R^n$ and $x$ is a critical point of $f$, i.e., $\nabla f = 0$. The Morse index $i(x)$ of $f$ at $x$ is defined as the number of negative eigenvalues of the Hessian $D^2 f$ counted with multiplicity and the nullity $\nu (x)$ is the dimension of kernel at $x$. Namely,
	\begin{align}
		i_f(x) &:= \max \left\{\operatorname{dim} V\ |\  V \subset \mathbf{R}^{n}\; \mbox{is a subspace}, v^{T} D^2f(x) v<0, \forall v \in V \backslash\{0\}\right\}, \nn \\
		\nu_f(x) &:= \dim \ker D^2 f(x). \nn
	\end{align}
	If $\nu_f(x) = 0$, then $f$ is called non-degenerate at $x$.
\end{definition}

Via the Morse index, the critical points of $f$ can be classified into the {\it local maximum point} whose Morse index is $n$, the {\it local minimum point} whose Morse index is $0$ and the {\it saddle point} whose Morse index is between $0$ and $n$. The sets of above classification are denoted by $\cC_n(f)$, $\cC_0(f)$ and $\cC_s(f)$ respectively. We refer readers to \cite{audin2014morse} and \cite{milnor2016morse} for more details on Morse theory.

Let
\begin{align}
	\bar{U}(x) = \sum_{i=1}^{n}\frac{1}{4} x_{i}^{4}-\frac{\aa^2}{2}x^Tx. \lb{eqn:bar.U}
\end{align}
Then $U(x) = \bar{U}(x)+  \frac{\bb}{2}x^Tx- \frac{1}{2} x^T S x$ by \eqref{eqn:pote}.
Denote by $\cC(U)$ and $\cC(\bar{U})$ the sets of all critical points of $U(x)$ and $\bar U(x)$ respectively. It is direct to obtain that
\begin{align}
	\cC(U)=\{x\in \R^n |\  (x_1^3 -\aa^2x_1,\dots,x_n^3-\aa^2x_n)^T -Sx =0\}, \lb{eqn:grad.U}
\end{align}
and
\begin{align}
	\cC(\bar U)=\{x\in \R^n |\ (x_1^3 -\aa^2x_1,\dots,x_n^3 -\aa^2x_n)^T = 0\}. \lb{eqn:grad.b.U}
\end{align}
Both $\cC(\bar{U})$ and $\cC(U)$ are non-empty since $0 \in \cC(U) \cap \cC(\bar{U})$.
Define the set
\begin{align}
	\A = \{-\aa, \aa\},\; \text{and}\quad \A_0 = \{-\aa,0,\aa\}.
\end{align}
Then we define
\begin{align}
	\A_0^n = \{-\aa, 0, \aa\}^n:= \overbrace{\{-\aa, 0, \aa\} \times \{-\aa, 0, \aa\}\times \cdots \{-\aa, 0, \aa\}}^{n}.
\end{align}
Also, $\A^n  = \{-\aa,\aa\}^n$.
Via the Morse index, the critical points of $\bar{U}(x)$ can be classified as follows.

\begin{lemma}\lb{lem:mor.bar.u}
	When $\aa > 0$, $\cC(\bar{U}) = \A_0^n$. More precisely,
	\begin{enumerate}[label=\roman*)]
		\item 	the only local maximum point is the origin, i.e., $\cC_n(\bar{U}) = \{0\}$;
		\item  the set of local minimum points is  $\cC_0(\bar{U}) =\A^n$;
		\item the set of saddle points is $\cC_s(\bar{U}) = \cC(\bar{U})\bs \l(\cC_n(\bar{U}) \cup\cC_0(\bar{U}) \r)$.
	\end{enumerate}
	Moreover, for $x\in \cC(\bar{U})$, we have
	\begin{align}
		i_{\bar{U}}(x) = \#\{j|x_j = 0, x_j \; \text{is the $j$th component of}\; x\}.
	\end{align}
\end{lemma}

\begin{proof}
	Solve  $\nabla \bar{U}(x) = 0$ in \eqref{eqn:grad.b.U} directly and obtain that the roots are given by $x = (x_1, \dots, x_n)$ with $x_i \in \A_0$. Therefore, the number of the critical points of $\bar U(x)$ is $3^n$.
	
	The Hessian of $\bar{U}$ is given by
	$D^2 \bar{U}(x) = \diag\{3x_1^2 -\aa^2,\dots, 3 x_n^2 -\aa^2\}$. If $x_i = 0$, then $3x_i^2 -\aa^2 < 0$; if $x_i = \pm \aa$, then $3x_i^2 - \aa^2> 0$. Its Morse index is given by the number of $x_i = 0$.
	
	Therefore, the origin is the unique local maximum point; $x \in \A^n$ is the local minimum points. The rests are saddles, namely at least one $x_i=0$ and at least one $x_j \in  \A$.
\end{proof}

Note that $\cC_0(\bar U) = \A^n$ where $\A= \{-\aa, \aa\}$. Recall that candidates of Ising model is $C(E) = \{-1,1\}^n$. Via the signum map, the following holds.

\begin{corollary}\lb{lem:bar.U.sign.vec}
	$\{\sgn(x)|x\in \cC_0(\bar U)\} = C(E)$.
\end{corollary}

For critical points of $U(x)$, we have an a priori estimate as follows.

\begin{lemma}\lb{lem:x.div.p.0}
	For any $\ep >0$, there exists an $\aa_1$ such that for any $\aa> \aa_1$ and $x \in  \cC(U)$,
	\begin{align}
		\frac{|x|}{\aa^2}< \ep.
	\end{align}
\end{lemma}

\begin{proof}
	Each $x\in \cC(U)$ satisfies that
	\begin{align}\lb{eqn:devi.p3}
		(x_1^3,\dots, x_n^3)^T
		=\aa^2 x	+(S-\bb I_n)x.
	\end{align}
	Dividing by $1/\aa^6$, we have \eqref{eqn:devi.p3} can be rewritten as
	\begin{align}
		z^3_i=\frac{z_i}{\aa^2}+\frac{1}{\aa^4}\sum _{j=1}^n m _{ij}z_j, \lb{eqn:K.z.3}
	\end{align}
	where $z_i := x_i / \aa^2$ and $(m _{ij})_{n\times n} := S-\bb I_n$ for $1\leq i\leq n$.
	Arbitrarily choose one increasing sequence $\{\tilde \aa_k\}_{k=1}^{\infty}$ satisfying $\lim_{k\to \infty} \td \aa_k = \infty$. For all $i$, rewrite $z_i$ as $z_i(\td \aa_k)$.
	Either at least one $x\in \cC(U)$ and a sub-sequence of $\{\td \aa_k\}_{k=1}^{\infty}$ exists which denoted again by $\{\td \aa_k\}_{k=1}^{\infty}$ for simplicity such that $\{|z_i(\td \aa_k)|\}_{k = 1}^{\infty}$ is unbounded or for all $x \in \cC(U)$, $|z_i(\td \aa_k)|$ are bounded by a positive number $B_{1}$.  	
	
	Suppose that $|z_i(\td \aa_k)|$ is unbounded and
	$|z_\ell (\td \aa_j)|:= \max _{1\leq i\leq n}|z_i(\td \aa_j)|>j$.
	For each given $k$, by \eqref{eqn:K.z.3},
	\begin{align}
		1=\frac{1}{\td\aa_j ^2z_\ell^2 (\td \aa_j)}+\frac{1}{\aa_j^4}\sum _{i=1}^nm _{\ell i}\frac{z_i(\td\aa_j)}{z_\ell^3 (\td \aa_j)}.\lb{eqn:K.td.0}
	\end{align}
	It is a contradiction that the left hand side of \eqref{eqn:K.td.0} is a constant while the right hand side of \eqref{eqn:K.td.0} converges to zero when $j$ tends to infinity. Then all $|z_i(\aa)|$ are bounded by a positive constant $B_1$.
	
	Suppose $|z_i(\aa)| < B_1$ for all $i$ and $\aa >0$. It yields that
	\begin{align}
		\lim _{\aa\rightarrow \infty }z^3_i(\aa)=\lim _{\aa\rightarrow \infty }\left (\frac{z_i(\aa)}{\aa^2}+\frac{1}{\aa^4}\sum _{j=1}^nm _{ij}z_j(\aa)\right )=0,
	\end{align}
	which implies that $\lim_{\aa \to \infty} z_i(\aa) = 0$. The proof is completed.
\end{proof}

When $\aa$ is large enough, each critical point of $U(x)$ can be approximated by a unique critical point of $\bar{U}(x)$ as follows.

\begin{proposition}\lb{thm:ep.net}
	Let $\aa_1$ be as in Lemma \ref{lem:x.div.p.0}.
	For any given positive constant $B_2$, there exists an $\aa_2 > \aa_1$ such that for any $\aa > \aa_2$ and every $x\in \cC(U)$, there exists one $\bar{x} \in \cC(\bar{U})$ satisfying
	\begin{align}\lb{eqn:x.barx.bdd}
		|x-\bar{x}| < \frac{B_2}{\aa},
	\end{align}
	and $\bar x$ is uniquely determined by $x$.
	Furthermore, $i_U(x) = i_{\bar{U}}(\bar{x})$.
\end{proposition}

\begin{proof}
	Since $\aa \neq 0$, we have that
	\begin{align}
		\frac{1}{\aa^2}\nabla U(x) = &\frac{1}{\aa^2}(x_1^3, \dots, x_n^3)^T
		-\l(I_n-\frac{\bb}{\aa^2}I_n+\frac{S}{\aa^2}\r)(x_1, \dots, x_n)^T
		=0.		
	\end{align}
	According to Lemma \ref{lem:x.div.p.0}, for any given $\ep > 0$, there exists an $\aa_3 > \aa_1$ such that for all $\aa \geq \aa_3$,
	\begin{align}
		\l|\frac{\bb}{\aa^2} x_i - \l(\sum_{j = 1}^{n}s_{ij}\frac{x_j}{\aa^2}\r)\r| < \ep, \lb{eqn:xi.ep}
	\end{align}
	for all $i$.
	Let $f_{\pm,\ep}(x)= \frac{1}{\aa^2}x^3 -x \pm \ep$.  	
	Suppose  $\aa > \aa_4 := \sqrt[4]{27 \ep^2/4}$. Note that $f_{+,\ep}(x) = 0$ (resp. $f_{-,\ep}(x) = 0$)  possesses three real roots, denoted by $x_{+,i}$ (resp. $x_{-,i}$) where $i=1, 2, 3$.
	Then for all $i$,
	\begin{align}
		f_{-,\ep}(x_i)< \frac{1}{\aa^2}x_i^3 -x_i +\frac{\bb}{\aa^2} x_i + \l(\sum_{j = 1}^{n}s_{ij}\frac{x_j}{\aa^2}\r) <f_{+,\ep}(x_i). \lb{eqn:f.K.f}
	\end{align}
	For any $\lm\in \R$, if $f_{-,\ep}(\lm)<0$ and $f_{+,\ep}(\lm) > 0$, then $\lm \in (x_{+,1} , x_{-,1}) \cup (x_{-,2},x_{+,2}) \cup (x_{+,3},x_{-,3})$. Hence, every $x_i$ of the critical point $x$ satisfies
	\begin{align}
		x_i \in (x_{+,1} , x_{-,1}) \cup (x_{-,2},x_{+,2}) \cup (x_{+,3},x_{-,3}).
	\end{align}
	
	{\bf Claim 1. }{\it For any given $\ep > 0$, $f_{+,\ep}(x) = 0$ (resp. $f_{-,\ep}(x) = 0$) possesses three solutions $x_{+,i}$ (resp. $x_{-,i}$) for $i=1, 2, 3$ and $x_{+,i}$ (resp. $x_{-,i}$) satisfies}
	\begin{align}
		&\lim_{\aa \to \infty} |x_{+,1}+\aa| =\lim_{\aa \to \infty} |x_{+,3}-\aa| =  \frac{\ep}{2},\; \lim_{\aa \to \infty} |x_{+,2}| = \ep.\nn \\
		(\text{resp.} \quad &\lim_{\aa \to \infty} |x_{-,1}+\aa| =\lim_{\aa \to \infty} |x_{-,3}-\aa| =  \frac{\ep}{2},\; \lim_{\aa \to \infty} |x_{-,2}| = \ep. )\nn
	\end{align}
	
	We leave the tedious proof in Appendix.
	By Claim 1, $\lim_{\ep \to 0} \aa_4 = 0$ and the arbitrariness of $\ep$ in \eqref{eqn:xi.ep}, $x_i$ satisfies one of the following
	\begin{align}
		\lim_{\aa \to \infty} |x_i-\aa| =0, \; \lim_{\aa \to \infty} |x_i+\aa| = 0, \; \lim_{\aa \to \infty} |x_{i}| = 0,
	\end{align}
	for any $i$.
	Therefore,  for any $\ep > 0$, there exists an $\aa_5 > \aa_3$ such that for all $\aa \geq \aa_5$,
	\begin{align}
		\l|\frac{\bb}{\aa^2} x_i - \l(\sum_{j = 1}^{N}s_{ij}\frac{x_j}{\aa^2}\r)\r| < \frac{\ep}{\aa}. \lb{eqn:xi.ep.2}
	\end{align}
	
	Let $\td f_{\pm,\ep}(x) =\frac{1}{\aa^2}x^3-x \pm\frac{\ep}{\aa}$.

	{\bf Claim 2. }{\it For any given $\ep > 0$, $\td f_{+,\ep}(x)  = 0$  (resp. $\td f_{-,\ep}(x) = 0$) possesses three solutions $\td x_{-,i}$ (resp. $\td x_{-,i}$) for $1 \leq i \leq 3$. There exist $B_2$ and $\aa_6$ such that for all $\aa > \aa_6$,  $\td x_{+,i}$  (resp. $\td x_{-,i}$) satisfies that}
	\begin{align}
		&|\td x_{+,1}+\aa| < \frac{B_2}{\aa},\quad |\td x_{+,2}| < \frac{B_2}{\aa},\quad |\td x_{+,3}-\aa| < \frac{B_2}{\aa}.\nn \\		
		(resp. \quad &	|\td x_{-,1}+\aa| < \frac{B_2}{\aa},\quad |\td x_{-,2}| < \frac{B_2}{\aa},\quad |\td x_{-,3}-\aa| < \frac{B_2}{\aa}.)\nn
	\end{align}
	
	We also leave the tedious proof in Appendix.	
	By Claim 2,  $x_i$ satisfies one of following inequalities for all $\aa > \aa_6$,
	\begin{align}
		|x_i-\aa| < \frac{B_2}{\aa},\quad |x_i| < \frac{B_2}{\aa},\quad |x_i+\aa| < \frac{B_2}{\aa}.\lb{eqn:1.of.3}
	\end{align}
	By \eqref{eqn:1.of.3}, we define that $\bar v := \lim_{\aa \to \infty} x/\aa \in \{-1,0,1\}^n$. Let $\bar{x} := \aa v$ with $\bar{x}\in \A_0^n$ which is uniquely determined by $x$.
	With out loss of generality, assume $i_{\bar{U}}(\bar x) = i_0$. By \eqref{eqn:1.of.3}, the number of $x_i$s satisfying $|x_i-\aa| < \frac{B_2}{\aa}$ or $|x_i+\aa| < \frac{B_2}{\aa}$ is $i_0$ and the number of $x_i$s satisfying $ |x_i| < \frac{B_2}{\aa}$ is $n-i_0$ when $\aa> \aa_6$.	
	
	The Hessian of $U(x)$ is given by
	\begin{align}
		D^2 U =\diag \{3x_1^2 + \bb-\aa^2, \dots, 3x_n^2 + \bb-\aa^2\} -S.
	\end{align}
	Decompose $D^2 U$ as the sum of $\diag \{3x_1^2 -\aa^2, \dots, 3x_n^2 -\aa^2\}$ and $\bb I_n-S$.
	Suppose that the eigenvalues of $\frac{1}{\aa^2}D^2 U$ are
	\begin{align}
		\sg(\frac{1}{\aa^2}D^2 U) = \{\lm_1, \dots, \lm_n\},
	\end{align}
	where $\lm_1 \geq \cdots \geq \lm_n$. Furthermore, suppose that
	\begin{align}
		&\sg (\frac{1}{\aa^2}\diag \{3x_1^2 -\aa^2, \dots, 3x_n^2 -\aa^2\}) = \{\bar{\lm}_1, \dots, \bar{\lm}_n\},\\
		&\sg((\bb I_n-S)/\aa^2) =\{\mu_1, \dots, \mu_n\},
	\end{align}
	where $\bar{\lm}_1 \geq \cdots \geq  \bar{\lm}_n$ and $\mu_1 \geq \cdots \geq  \mu_n$.
	Since $(\bb I_n-S)$ is a constant matrix and $i_{\bar{U}}(\bar{x}) = i_0$, there exists an $\aa_7> \aa_6$ such that $\max\{|\mu_1|, |\mu_n|\} <1/3 $, $\bar{\lm}_{i} > 5/3$ for $1\leq  i \leq i_0$ and $\bar{\lm}_{i}< -2/3$ for $i_0 +1\leq  i \leq n$.
	According to the Weyl's inequality (cf. \cite[Theorem 4.3.1]{horn2012matrix}), $\lm_i$ satisfies that
	$\bar{\lm}_i +  \mu_n <\lm_i < \bar{\lm}_i + \mu_1$. Therefore, $\lm_i$ possesses the same sign as $\bar{\lm}_i$. Then $i_{U}(x) = i_0$ and the critical points of $U(x)$ are all non-degenerate.
	
	Let $\aa_2 = \max_{3\leq i\leq 7}\{\aa_i\}$. Then the proposition follows.
\end{proof}

As $\aa$ tends to infinity, every critical point $x$ of $U(x)$ satisfies $|x -\bar x| \to 0$ by \eqref{eqn:x.barx.bdd}.
According to Proposition \ref{thm:ep.net}, when $\aa > \aa_2$, $x$
can be written as
$x = \bar{x} +\dl$ where $\bar{x} \in \cC(\bar{U})$ and $|\dl| \sim \O(1/\aa)$.

\begin{corollary}\lb{coro:minimum}
	Suppose $x = (x_1,\dots,x_n)^T$ is one local minimum point of $U(x)$. There exists $\aa_8 > \aa_2$ such that following statements hold.
	\begin{enumerate}[label=\roman*)]
		\item $x= \bar{x} + \dl$ where $\bar{x} \in \A^n$ and $\dl = (\dl_1, \dots, \dl_n)$ with $\dl_i\sim \O(1/\aa)$; \lb{coro:it.i}
		\item  $x_i \neq 0$ for all $i\in \{1,\dots, n\}$;\lb{coro:it.ii}
		\item   $\mathrm{sgn}(x) = \mathrm{sgn}(\bar{x})$.\lb{coro:it.iii}
	\end{enumerate}
\end{corollary}

\begin{proof}
	Since $x$ is a local minimum point, both $i_U(x) = 0$ and $i_{\bar{U}}(\bar x) = 0$. By Lemma \ref{lem:mor.bar.u}, for each $i$, $\bar{x}_i$ of $\bar{x}$ satisfies $|\bar x_i| = \aa$. By \eqref{eqn:x.barx.bdd}, one can choose a proper $\aa_8 > \aa_2$ such that $\bar{x}_i +\dl_i  \neq 0$ and $\sgn(x_i) = \sgn(\bar{x}_i)$ when $\aa > \aa_8$ for all $i$.
\end{proof}

\begin{proposition}\lb{prop:3n.cp}
	For any given $\bb$ and $S$, there exists a sufficiently large $\aa_0 > \aa_8$ such that when $\aa > \aa_0$,
	\begin{enumerate}[label = \roman*)]
		\item $U(x)$ possesses $3^n$ critical points; \lb{prop:it.i}
		\item $U(x)$ possesses $2^n$ local minimum points;\lb{prop:it.ii}
		\item  $\{\sgn(x)|x\in \cC(\bar U)\} =C(E)$.
		\lb{prop:it.iii}
	\end{enumerate}
\end{proposition}

\begin{proof}
	Note that
	$\bar U(x)$ possesses $3^n$ critical points. 	
	By Proposition \ref{thm:ep.net}, $D^2 U$ at the critical points is non-degenerate and the critical points of $U(x)$ tend to the critical points of $\bar{U}(x)$ by \eqref{eqn:x.barx.bdd}.
	By the non-degeneracy, for every $\bar x\in \cC(\bar U)$, there exists a uniform $\dl$ and a uniform $\aa_9 > \aa_8$ such that only one $x\in \cC(U)$ satisfies $|x -\bar{x}| < \dl$ for all $\aa > \aa_9$. Therefore, the number of critical points of $U(x)$ is at most $3^n$.
	
	Each row of $\nabla U = 0$ can be regarded as $f_i = 0$ where $ f_i(x) = \frac{1}{\aa^2}x^3 -x +\frac{\bb}{\aa^2} x + (\sum_{j = 1}^{N}s_{ij}\frac{x_j}{\aa^2})$. By \eqref{eqn:xi.ep.2}, there exists a sufficiently large constant $\aa_{10}$ such that when $\aa > \aa_{10}$, $f_i(-2\aa) < 0$, $f_i(-\aa/2) >0$, $f_i(\aa/2) <0$ and $f_i(2\aa) > 0$ hold. Therefore, $f_i(x)= 0$ possesses at least three roots for all $1\leq i \leq n$.
	
	Hence, when $\aa > \aa_0 := \max\{\aa_9, \aa_{10}\}$ , $\nabla U = 0$ possesses $3^n$ roots. It yields \ref{prop:it.i} of this proposition holds.
	
	By Proposition \ref{thm:ep.net}, $U(x)$ possesses $2^n$ local minimum points because $\bar U(x)$ possesses $2^n$ local minimum points. Then \ref{prop:it.ii} of this proposition follows.
	
	By \ref{coro:it.iii} of Corollary \ref{coro:minimum}, the signum vectors of local minimum points of $U(x)$ are the same as the ones of local minimum of $\bar U(x)$. By Lemma \ref{lem:bar.U.sign.vec}, \ref{prop:it.iii} of this proposition holds.
\end{proof}

Define that the maximum value of $U(x)$ among the minimum points and the lowest value of $U(x)$ energy among saddles as
\begin{align}
	U_{M} = \max_{x \in \cC_0(U)} U(x), \quad 	U_{s} &= \min_{x \in \cC_s(U)}U(x). \lb{eqn:Us}
\end{align}

\begin{lemma}
	There exists $\aa_{11} > \aa_0$, such that for any given $\aa > \aa_{11}$, if $x_s\in \cC_s(U)$ satisfying $U(x_{s})= U_{s}$, its Morse index $i(x_s)$ is $1$.
\end{lemma}

\begin{proof}
	Since $x_s$ is a saddle, we have that $i_U(x_s) \geq 1$.
	Suppose that $\bar{x}_s$ is the critical point of $\bar{U}(x)$ in the $1/\aa$-neighborhodd of $x_s$. We write $x_s = \bar{x}_s + \dl_s$ when $\aa > \aa_0$. Suppose $i_U(x_s) = j$. So $i_{\bar{U}}(\bar{x}_s) = 1$. Then $U(x_s)$ is given by
	\begin{align}
		\lb{eqn:U.x.s}
		U(x_s)  =&\sum_{i = 1}^n\frac{1}{4}(\bar{x}_{s,i} +\dl_{s,i})^4 +\sum_{i = 1}^n \frac{\bb-\aa^2}{2}(\bar{x}_{s,i} +\dl_{s,i})^2 \\
		& \quad -\frac{1}{2}(\bar{x}_{s} +\dl_{s})^TS(\bar{x}_{s} +\dl_{s}) \nn\\
		=& -\frac{(n-j)\aa^4}{4} + \O(\aa^2),  \nn
	\end{align}
	where the last equality holds because $|\bar{x}_{s,i}| = \aa$ and $\dl_{s,i} \sim \O(1/\aa)$.
	Then there exists $\aa_{11} > \aa_0$ such that if $U(x_{s})= U_{s}$, then $i_{U}(x_s) = 1$.
\end{proof}

\begin{proposition}\lb{prop:Us>um}
	There exists $\aa_{12} > \aa_0$ such that $U_s > U_M$ when $\aa > \aa_{12}$.
\end{proposition}
\begin{proof}
	For any given $\aa$, suppose that  $x_{m}$ satisfies that $U(x_{m})= U_{M}$ and correspondingly $\bar{x}_{m}$ is the minimum point of $\bar{U}(x)$ with $x_{m} = \bar{x}_{m} +\dl_{m}$. Suppose that $x_{s}$ satisfies $U(x_{s})= U_{s}$ and $\bar{x}_{s}$ is the saddle of $\bar{U}(x)$ with $x_{s} = \bar{x}_{s} +\dl_{s}$.
	Therefore $U(x_s) = -\frac{(n-1)\aa^4}{4} + \O(\aa^2)$ by \eqref{eqn:U.x.s}.
	Note that  $U(x_m)$ satisfies
	\begin{align}
		U(x_m) =&  \sum_{i = 1}^n\frac{1}{4}(\bar{x}_{m,i} +\dl_{m,i})^4 + \frac{\bb-\aa^2}{2}(\bar{x}_{m,i} +\dl_{m_i})^2 \\
		& \quad -\frac{1}{2}(\bar{x}_{m} +\dl_{m})^TS(\bar{x}_{m} +\dl_{m}) \nn \\
		=&	-\frac{n\aa^4}{4} + \O(\aa^2),	\nn
	\end{align}
	where last equality holds because $|\bar{x}_{m,i}| = \aa$ and $\dl_{m,i} \sim \O(1/\aa)$.
	By \eqref{eqn:U.x.s}, it follows that
	\begin{align}
		U_s-U_M = U(x_s)-	U(x_m) =& \frac{\aa^4}{4} + \O(\aa^2).
	\end{align}
	There exists $\aa_{12}$ with $ \aa_{12} > \aa_0$ such that $U_s-U_M  > 0$.
\end{proof}

Let $N := 2^n$. For any given Ising model $E(v)$, we assume that
\begin{align}
	E(v_1)= E(v_2)\leq E(v_3) =E(v_4) \leq \cdots \leq E(v_{N-1}) =  E(v_N),
\end{align}
where $v_i\in C(E)=\{-1,1\}^n$ and $v_{2i-1} = -v_{2i}$. Denote that
\begin{align}
	d_i = E(v_{2i}) - E(v_{2i-2}) \geq 0, \quad \; \text{where}\; i \in \{1, \dots, N/2\}. \lb{eqn:di}
\end{align}
We can label the local minimum points of $\bar U(x)$ by $\bar x_i$ with $\bar x_i = \aa v_i$. Since $|\bar x_i| = |\bar x_j|$ for any $\bar x_i, \bar x_j\in \cC_0(\bar U)$, we have that $\bar{U}(\bar x_i) = \bar{U}(\bar x_{j})$. For $i\in \{2,\dots, N/2\}$, we have that
\begin{align}
	U(\bar x_{2i}) - U(\bar x_{2i-2}) = & \l(\bar{U}(\bar x_{2i}) + \frac{\bb}{2} |\bar x_{2i}|^2 -\frac{1}{2}\bar x_{2i}^T S \bar x_{2i} \r) \\
	&\quad - \l(\bar{U}(\bar x_{2i-2}) + \frac{\bb}{2} |\bar x_{2i-2}|^2 - \frac{1}{2}\bar x_{2i-2}^T S \bar x_{2i-2} \r) \nn \\
	=& \frac{1}{2}\l(\bar x_{2i-2} S \bar x_{2i-2}^T - \bar x_{2i} S \bar x_{2i}^T\r) \nn \\
	=& \aa^2 d_i, \nn
\end{align}
where the last equality holds by $\bar x_{2i} = \aa v_{2i}$. It follows that
\begin{align}
	U(\bar x_1)= U(\bar x_2) \leq U(\bar x_3)= U(\bar x_4) \leq \cdots \leq U(\bar x_{N-1}) = U(\bar x_N). \lb{eqn:U.seq.bar.x}
\end{align}
By Lemma \ref{thm:ep.net}, we also label the local minimum points of $U(x)$ by $x_i$ for $1\leq i \leq N$ such that $x_i$ satisfies
\begin{align}
	|x_i - \bar x_i| < B/\aa, \lb{eqn:lb}
\end{align}
when $\aa > \aa_{12}$.
\begin{lemma}\lb{lem:MI}
	Given $x_i \in \cC_0(U)$, there exist $M_i>0$ and $A_{i} > \aa_{12}$ such that when $\aa > A_i$,
	\begin{align}
		|U(x_i) - U(\bar x_i)| < M_i. \lb{eqn:M_i}
	\end{align}
\end{lemma}
\begin{proof}
	By Corollary \ref{coro:minimum}, we have that $x_i= \bar x_i +\dl_i$ where $x_i=(x_{i,1}, \dots, x_{i,n})^T$, $\bar x_i=(\bar x_{i,1}, \dots, \bar x_{i,n})^T \in \cC_0(\bar{U})$, and $\dl_i =(\dl_{i,1}, \dots, \dl_{i,n})^T$. It follows that
	\begin{align}
		U(x_i) - U(\bar{x}_i) =& \sum_{j = 1}^{n}  (\bar{x}_{i,j}^3\dl_{i,j} +\frac{3}{2}\bar{x}_{i,j}^2\dl_{i,j}^2+\bar{x}_{i,j}\dl_{i,j}^3 + \frac{1}{4}\dl_{i,j}^4) \\
		&\quad + \sum_{j = 1}^{n} \frac{\bb-\aa^2}{2}(2\bar{x}_{i,j}\dl_{i,j}+\dl_{i,j}^2) - \dl^TS \bar{x}-\frac{1}{2}\dl^TS \dl \nn \\
		=&\sum_{j = 1}^{n} (\aa^2\dl_{i,j}^2 + \bb\dl_{i,j}\bar{x}_{i,j}+ \bar{x}_{i,j}\dl_{i,j}^3+\frac{1}{4}\dl_{i,j}^4+\frac{\bb}{2}\dl_{i,j}^2) - \dl^TS \bar{x}-\frac{1}{2}\dl^TS \dl \nn,
	\end{align}
	where the second equality holds by $|\bar{x}_{i,j}| = \aa$. For each $j$, the terms $\aa^2\dl_{i,j}^2$, $\bb\dl_{i,j}\bar{x}_{i,j}$ and $\dl^T S \bar{x}$ are all bounded because $|\bar x_{i,j}| = \aa$ and $|\dl_{i,j}| <B_2/\aa$. The terms $\bar{x}_{i,j}\dl_{i,j}^3$, $\frac{1}{4}\dl_{i,j}^4$, $\frac{\bb}{2}\dl^2$ and $\frac{1}{2}\dl^TS \dl$ tend to zero as $\aa$ tends to infinity. Therefore, there exist $A_i >\aa_{12}$ and $M_i > 0$ such that for all $\aa \geq A_i$,
	\begin{align}
		|U(x_i) - U(\bar{x}_i) | < M_i.
	\end{align}
	The proof is complete.
\end{proof}

Note that $d_i$ in \eqref{eqn:di}, $A_i$ in Lemma \ref{lem:MI}, and $M_i$ in \eqref{eqn:M_i} only depend on $\bb$ and $S$.
Let $d_{\min}:= \min\{d_i|d_i \neq 0, 1\leq i\leq N/2\}$, $A_{\max}:= \max\{A_i|1\leq i \leq N\}$, and  $M:=\max \{M_i| 1\leq i \leq N\}$.
We prove Theorem \ref{thm:main} as follows.

\begin{proof}[Proof of Theorem \ref{thm:main}]
	Suppose $x_0$ is a global minimum point of $U(x)$. Let $v_0 = \sgn (x_0)\in C(E)$.
	By Lemma \ref{lem:MI}, when $\aa > A_{\max}$,
	\begin{align}
		|U(x_i) - U(\bar x_i)| < M, \quad \forall\; 1\leq i\leq N. \lb{eqn.M1}
	\end{align}
	Let
	\begin{align}
		\aa_* := \max\l\{A_{\max}, \sqrt{\frac{3M}{d_{min}}}\r\}.
	\end{align}
	Note that $d_{\min}\neq 0$. Assume by contradiction that there is $v' \in C(E)$ such that $E(v') < E(v_0)$. Then
	\begin{align}
		E(v_0) -E(v') > d_{\min} >0.
	\end{align}
	Let $\bar x_0 = \aa v_0$ and $\bar x' = \aa v'$.
	When $\aa> \aa_*$, we have that
	\begin{align}
		U(\bar x_{0})- U(\bar x')  > \aa_*^2 d_{\min} \geq 3M.
	\end{align}
	Together with \eqref{eqn.M1}, we have that $|U(x_0) - U(\bar x_0)| < M$ and $|U(x') - U(\bar x')| < M$. It follows that when $\aa > \aa_*$,
	\begin{align}
		U(x_0) - U(x') > M >0.
	\end{align}
	which contradicts that $x_0$ is the global minimum point of $U(x)$. Then $v_0$ is a minimizer of $E(v)$.
	The proof is complete.
\end{proof}

If we further assume that the Ising model satisfies
\begin{align}
	E(v_1)= E(v_2) < E(v_3) =E(v_4) < \cdots < E(v_{N-1}) =  E(v_N),
\end{align}
then following a similar argument as \eqref{eqn:U.seq.bar.x}, we obtain that $d_i \neq 0$ for $1\leq i\leq N/2$, and
\begin{align}
	U(\bar x_1)= U(\bar x_2) < U(\bar x_3) =U(\bar x_4) < \cdots < U(\bar x_{N-1}) =  U(\bar x_N). \nn
\end{align}
As \eqref{eqn:lb}, we also label local minimum points of $U(x)$ as $x_1, x_2, \dots, x_n$ satisfying \eqref{eqn:lb}. Without losing the generality, we assume that $U(x_{2i-1}) \leq U(x_{2i})$ for all $1 \leq i \leq N/2$ by Lemma \ref{lem:MI}.
Following a similar argument as the proof of Theorem \ref{thm:main}, when $\aa > \aa_*'$, there exists $\aa_*'$ such that $U(\bar x_{2i}) - U(\bar x_{2i-2}) > 6M$ for $2\leq i\leq N/2$. Because $|U(\bar x_{i})- U(x_{i})|< M$ and $U(x_{2i-1}) \leq U(x_{2i})$ for all $1 \leq i \leq N/2$, we have that following corollary holds. We omit the detailed proof.
\begin{corollary}
	Suppose that the Ising model $E$ satisfies
	\begin{align}
		E(v_1)= E(v_2) < E(v_3) =E(v_4) < \cdots < E(v_{N-1}) =  E(v_N), \nn
	\end{align}
	where $v_i\in C(E)$ and $v_{2i-1} = -v_{2i}$. There exists $\aa_*'> 0$ such that when $\aa > \aa_*'$, the minimum points of $U(x)$ satisfy that
	\begin{align}
		U(x_1) \leq  U(x_2) < U(x_3) \leq U(x_4) < \cdots < U(x_{N-1}) \leq   U(x_N), \nn
	\end{align}
	where $x_i \in \cC_0(U)$ and
	\begin{align}
		\{\sgn(x_{2i-1}), \sgn(x_{2i})\} = \{v_{2i-1}, v_{2i}\}, \quad\; \forall i \in \{1,\dots, N/2\}. \nn
	\end{align}
\end{corollary}

\subsection{Examples of Ising model in $\R^2$ and $\R^3$}\lb{sec:opt.app}
When $n = 2$, the Ising energy can be reduced to $E = -\frac{1}{2}v^T Sv$
with $S= S_2 = (\begin{smallmatrix}
	0 & 1 \\1 & 0
\end{smallmatrix})$ whose eigenvalues are given by $\sg(S) = \{-1, 1\}$ and $v\in \{-1,1\}^2$.
Assume that $\bb > 1$.
The critical points of $U(x)$ in \eqref{eqn:pote} are given by the solutions of
\begin{align}
	\nabla U = \begin{pmatrix}
		x_1^2 +(\bb-\aa^2) & -1\\
		-1 & x_2^2 +(\bb-\aa^2)
	\end{pmatrix}
	\begin{pmatrix}
		x_1\\
		x_2
	\end{pmatrix}
	= 0.
\end{align}
For sake of simplicity, we define
\begin{align}
	\lm_1 =& \sqrt{\aa^2-\bb + 1},\lb{eqn:r2.lm1}\\
	\lm_2 =& \sqrt{\aa^2-\bb - 1},\lb{eqn:r2.lm2}\\
	\lm_3 =& \sqrt{(\aa^2 - \bb + \sqrt{(\aa^2 - \bb)^2-4})/2}, \lb{eqn:r2.lm3}\\
	\lm_4 =& \sqrt{(\aa^2 - \bb - \sqrt{(\aa^2 - \bb)^2-4})/2}.\lb{eqn:r2.lm4}
\end{align}
The number of critical points of $U(x)$ depends on $\aa$ and the bifurcation points of $\aa$ are $\sqrt{\bb-1}$, $\sqrt{\bb+1}$, and $\sqrt{\bb+2}$. Namely, the number of critical points and the local properties of the critical points change at those points which are given in Table \ref{tab:1} and shown in Figure \ref{fig:contourplot}.
In this case, when $\aa \geq \aa_* = \sqrt{\bb+2}$,
both $(-1,-1)$ and $(1,1)$ minimize Ising model;  $ (\lm_1, \lm_1)$ and $(-\lm_1, -\lm_1)$ minimize $U(x)$ in $\R^2$.

When $n \geq 3$, it is very involved to solve the critical points of $U(x)$ for any symmetric matrix $S$. without loss of generality, we take $S = S_3 = \l(\begin{smallmatrix}
	0 & 1 & -2\\
	1 & 0 & 3 \\
	-2 & 3 & 0
\end{smallmatrix}\r)$ as an example and suppose $\bb =10$. When $\aa = 0$, $U(x)$ possesses only one local minimum which is $(0,0,0)$. It follows that $v = (-1,1,1)$ or $(1, -1,-1)$ minimizes the Ising energy $E = -v_1v_2+2v_1v_3 - 3 v_2v_3$ with $v \in \{-1,1\}^3$. When $\aa > \sqrt{21.3} \approx 4.6$, $U(x)$ possesses $27$ critical points in $\R^3$. Therefore, one can choose that $\aa_* = 5$ in this case. Via the numerical computations, the global minimum points of $U(x)$ are $(-3.5, 3.7, 4.0)$ and $(3.5, -3.7, -4.0)$ whose signum vectors are $(-1,1,1)$ and $(1, -1,-1)$.

\begin{remark}
	One interesting phenomenon in above two examples is that the signum vectors of two local minimum points are the minimizers of Ising model when $\aa$ is between the first bifurcation point and the second bifurcation point in both cases of $S = S_2$ and $S=S_3$. When $S= S_2$, as Table \ref{tab:1}, when $\aa > \sqrt{\bb -1}$, the signum vectors of $(\lm_1,\lm_1)$ and $(-\lm_1,-\lm_1)$ give the minimizer of Ising in $\R^2$. When $S = S_3$ with $\bb -\aa^2= -3$, the numerical computations show $U(x)$ possesses only three critical points which are two local minimum points $(-0.19, 0.41,  0.50)$ and $(0.19, -0.41, -0.50)$, and one saddle $(0,0,0)$ when $\aa$ is between the first and the second bifurcation.  For $S =S_3$, the first two local minimum points correspond to the minimizer of Ising model via the numerical computations after the first bifurcation. It is open that whether this phenomenon exists for general Ising problems in $\R^n$.
\end{remark}

\setcounter{equation}{0}
\setcounter{figure}{0}
\section{Revisit some dynamical system algorithms for the Ising model}\lb{sec:app.mds}
In this section, we revisit some dynamical system algorithms for the Ising model. Using the mathematical mechanism founded in last section, we can understand the coherent Ising machines (CIM) in \cite{wang2013coherent}, the adiabatic Hamiltonian systems in \cite{yamamoto2017coherent},  the Kerr-nonlinear parametric oscillators (KPO) proposed in \cite{goto2019quantum} and simulation bifurcation (SB algorithm) proposed in \cite{goto2019combinatorial}.
In the following we use the original notations in their  papers.

\subsection{Coherent Ising machines}\lb{sec.cim}
To find minimizers of Ising model
\begin{align}
	\min _v \; E := -\frac{1}{2}v^T \Xi v, \lb{eqn:ising.xi}
\end{align}
where $v\in C(E) = \{-1,1\}^n$, $\Xi=(\xi_{ij})_{n\times n}$ is symmetric and $\xi_{ii} = 0$, a coherent Ising machine was proposed in (8) of \cite{wang2013coherent}  as
\begin{align}
	\begin{dcases}
		\dot{c}_{j}=\left(-1+p-\left(c_{j}^{2}+s_{j}^{2}\right)\right) c_{j}+\sum_{l=1, l \neq j}^{n} \xi_{j l} c_{l}, \\
		\dot{s}_{j}= \left(-1-p-\left(c_{j}^{2}+s_{j}^{2}\right)\right) s_{j}+\sum_{l=1, l \neq j}^{n} \xi_{j l} s_{l},
	\end{dcases}
	\lb{eqn:DOPO.sys1}
\end{align}
where $p > 1$ is a constant. If $(c, s)$ are classical solutions of \eqref{eqn:DOPO.sys1}, then $(c,s )\in C^2(\R, \R^{2n})$.
We define the function $U_{d} \in C^2(\R^{2n}, \R)$ as
\begin{align}
	U_{d}(c,s) &:= \sum_{j = 1}^n
	\l( \frac{1}{4}(c_{j}^{2}+s_{j}^{2})^2 -\frac{p}{2}(c_j^2-s_{j}^2) + \frac{1}{2}(c_j^2 +s_j^2)\r) - \frac{1}{2} c^T \Xi c - \frac{1}{2} s^T \Xi s.
\end{align}
Via direct computations, \eqref{eqn:DOPO.sys1} can be rewritten as
\begin{align}
	\begin{dcases}
		\dot{c}_{j}=-\frac{\pt U_d}{\pt c_j}, \\
		\dot{s}_{j}=-\frac{\pt U_d}{\pt s_j}.
	\end{dcases}
\end{align}
We further define the function $\td U_d := U_{d}(c,0)  \in C^2(\R^n, \R)$ as
\begin{align}
	\td U_d(c) := \sum_{i = 1}^n \l(\frac{1}{4}c_{i}^4 +\frac{1-p}{2}c_i^2\r) - \frac{1}{2}c^T\Xi c.
\end{align}
Denote  the set of critical points of $U_{d}(c,s)$ and the set of critical points of $\td U_{d}(c)$ as $\cC(U_d)$ and $\cC(\td U_d)$ respectively. Furthermore, define the projection map $\pi_d$ as
\begin{align}
	\pi_d: \; \cC(U_d) \to \cC(\td U_d), \quad	(c,s) \mapsto c.
\end{align}
Suppose that $\lm_{\Xi}$ is the largest eigenvalue of $\Xi$.

\begin{lemma}\lb{lem:DOPO}
	When $p >\lm_{\Xi}$, if $(c,s) \in \cC(U_d)$, then $s = 0$. Moreover, the map $\pi_d$ is well-defined and bijective and $i_{U_d}((c,0)) = i_{\td U_d}(c)$.
\end{lemma}

\begin{proof}
	Note that $\nabla U_d = 0$ is equivalent to
	\begin{align}
		\begin{dcases}
			\left(1-p+\left(c_{j}^{2}+s_{j}^{2}\right)\right) c_{j} - \sum_{l=1, l \neq j}^{n} \xi_{j l} c_{l} =0,\\
			\left(1+p+\left(c_{j}^{2}+s_{j}^{2}\right)\right) s_{j} - \sum_{l=1, l \neq j}^{n} \xi_{j l} s_{l} = 0.
		\end{dcases}
	\end{align}
	Since $p> \lm_{\Xi}$, $\nabla U_d = 0$ holds only if $s_j = 0$. 	
	Therefore, $\nabla U_d = 0$ can be reduced to
	\begin{align}
		\left(1-p+c_{j}^{2}\right) c_{j}-\sum_{l=1, l \neq j}^{n} \xi_{j l} c_{l}=0,\;  1\leq j\leq n. \lb{eqn:kpo.red}
	\end{align}
	Note that \eqref{eqn:kpo.red} is equivalent to $\nabla \td U_d = 0$. It yields that  $(c,0) \in \cC(U_d)$ if and only if $c \in \cC(\td U_d)$. Therefore, the map $\pi_d$ is a bijection between $\cC(U_d)$ and $\cC(\td U_d)$.
	
	The Hessian of $U_d$ at the critical points $(c,0)$ is given by
	\begin{align}
		D^2 U_d(c,0) =& \diag\{3 C, O_n\} + \diag\l\{(1-p)I_n,(1+p)I_n\r\} - \diag\{\Xi, \Xi\}\\
		=& \diag\{3 C+(1-p)I_n -\Xi,(1+p)I_n - \Xi\}, \nn
	\end{align}
	where $C = \diag\{c_1^2, \dots, c_n^2\}$, $O_n$ is an $n\times n$ matrix with all elements are $0$ and $c = (c_1, \dots, c_n)^T$.
	By direct computations,
	\begin{align}
		D^2 \td U_d(c) = 3 C+(1-p)I_n -\Xi.
	\end{align}	
	Note that $(1+p)I_n - \Xi$ is positively definite by $p > \lm_{\Xi}$.
	It follows $i_{U_d}((c,0)) = i_{\td U_d}(c)$. Then this lemma follows.
\end{proof}

Via Theorem \ref{thm:main}, we have following result.

\begin{proposition}\lb{prop:cim}
	When $p>\max\{\aa_*^2, \lm_{\Xi}\}$, if $(c,0)$ is the global minimum point of $U_d$,  $\sgn(c)$  minimizes Ising model $E = -\frac{1}{2}v^T \Xi v$.
\end{proposition}

\begin{proof}
	By Lemma \ref{lem:DOPO} and $U_d(c,0) = \td U_d(c)$, if $(c,0)$ minimize $U_d(c,s)$ globally, then $c$ minimizes $\td U_d$ globally when $p > \lm_{\max}$. If the parameters of $U(x)$ in \eqref{eqn:pote} satisfies $\aa = \sqrt{p}$, $\bb = 1$ and $S=\Xi$. The function $U = \td U_d $.
	Via Theorem \ref{thm:main}, it follows that if $(c,0)$ minimizes $U_d$ globally, the signum vector $\sgn(c)$ is a minimizer of Ising model \eqref{eqn:ising.xi} when $p>\max\{\aa_*^2, \lm_{\Xi}\}$.
\end{proof}

We consider another CIM which was proposed in \cite{yamamoto2017coherent}, as
\begin{align}
	\dot x = -\nabla U_c, \lb{eqn:cim}
\end{align}
where $U_c$ is given by
\begin{align}
	U_c (x) = \sum_{i = 1}^n \frac{1}{4} x_i^4 +\frac{1-p}{2} x_i^2 -\ep x^T S_c x,
\end{align}
where $p > 0$ and $0<\ep \ll 1$ and $S_c =(s_{ij})_{n\times n}$ is symmetric with $s_{ii} =0$.
To apply Theorem \ref{thm:main}, let $\bb = 1$, $\aa = \sqrt{p}$, and $S = 2\ep S_c$.
The function $U_c$ is the same as $U(x)$ given by \eqref{eqn:pote}.
So minimizing $E = -\frac{1}{2} v^T S v$ is equvilient to minimize $E = -\frac{\ep}{2} v^T S_c v$. Then we apply Theorem \ref{thm:main} directly and obtain the following result.

\begin{proposition}\lb{prop:cim2}
	When $p>\aa_*^2$, if $x$ is a global minimum point of $U_c(x)$, then $\sgn(x)$ is a minimizer of Ising model $E = -\frac{1}{2}x^T S_c x$.
\end{proposition}

\begin{remark}
	From the mathematical point of view, CIM in \eqref{eqn:DOPO.sys1} and \eqref{eqn:cim} are both designed to minimize $U_d$ globally via the gradient descent flow. A global minimum point yields a minimizer of the Ising model.	
	
	Readers may refer \cite{bohm2019poor, inagaki2016coherent} for the large-scale of CIMs and  \cite{haribara2016computational} for the measurement-feedback technique on CIM. With the help of the Brownian motion in CIMs, these algorithms show their power on solving the large-scale  combinatorial problems.
\end{remark}

\subsection{Adiabatic Hamiltonian systems}\lb{sec:hamil}
Suppose the Ising model
\begin{align}\lb{eqn:Ising,kpo}
	\min_{v} \; E := -\frac{1}{2} v^T J v,
\end{align}
where $v\in C(E) = \{-1,1\}^n$ and $J = (J_{i,j})_{n \times n}$ is a symmetric matrix with $J_{i,i}=0$. One adiabatic Hamiltonian system called KPO was introduced in \cite{goto2019quantum} as
\begin{align}
	\begin{dcases}
		\dot{x}_{i}=\frac{\pt H_k}{\pt y_i}=\left(p(t)+\Delta+K\left(x_{i}^{2}+y_{i}^{2}\right)\right) y_{i}-\xi_{0} \sum_{j=1}^{n} J_{i, j} y_{j}, \\
		\dot{y}_{i}=-\frac{\pt H_k}{\pt x_i}=-\left(p(t)-\Delta-K\left(x_{i}^{2}+y_{i}^{2}\right)\right) x_{i}+\xi_{0} \sum_{j=1}^{n} J_{i, j} x_{j}.
	\end{dcases}
	\lb{eqn:KPO.1}
\end{align}
The corresponding Hamiltonian $H_k(x,y,t)$ is
\begin{align}
	H_{k}(x,y,t) = & \sum_{i=1}^{n}\left(\frac{K}{4}\left(x_{i}^{2}+y_{i}^{2}\right)^{2}-\frac{p(t)}{2}\left(x_{i}^{2}-y_{i}^{2}\right)+\frac{\Delta}{2}\left(x_{i}^{2}+y_{i}^{2}\right)\right) \\
	&\quad  -\frac{\xi_{0}}{2} \sum_{i=1}^{n} \sum_{j=1}^{n} J_{i, j}\left(x_{i} x_{j}+y_{i} y_{j}\right), \nn
\end{align}
where $K, \Dl, \xi_0 > 0$ are constants and $p(t)> 0$ is a function of $t$ with $\dot{p}(t) > 0$.

For the same model \eqref{eqn:Ising,kpo}, another adiabatic Hamiltonian system called SB algorithm was introduced in \cite{goto2019combinatorial} as
\begin{align}
	\begin{dcases}
		\dot{x}_{i}= \frac{\pt H_s}{\pt y_i}= \Delta y_{i}, \\
		\dot{y}_{i}=-\frac{\pt H_s}{\pt x_i}=-\left(Kx_{i}^{2} + \Dl -p(t)\right) x_{i}+\xi_{0} \sum_{j=1}^{n} J_{i, j} x_{j}.
	\end{dcases}
	\lb{eqn:SB.1}
\end{align}
The corresponding Hamiltonian is
\begin{align}
	H_{s}(x,y,t) = & \sum_{i=1}^{n} \frac{\Dl}{2}y_{i}^{2} + \sum_{i=1}^{n} \l( \frac{K}{4}x_{i}^{4}+\frac{\Delta-p(t)}{2}x_{i}^{2}\r) -\frac{\xi_{0}}{2} x^T Jx, \lb{eqn:hami.SB.ex}
\end{align}
where $K, \Dl, \xi_0 > 0$ are constants, and $p(t) > 0$ is a function with $\dot{p}(t)> 0$.

The critical points of $H_k(x,y,t)$ and $H_s(x,y,t)$ are given by $\nabla H_k =0$ and $\nabla H_s =0$ respectively. However, the critical points of $H_k(x,y,t)$ and $H_s(x,y,t)$ are not the solutions \eqref{eqn:KPO.1} or \eqref{eqn:SB.1} because $\dot{p} \neq 0$. In the following, we take $p$ as a parameter to discussion the correspondence between the global minimum point of $H_k(x,y,t)$ and $H_s(x,y,t)$ with the minimizer of the Ising model \eqref{eqn:Ising,kpo}.
In Section \ref{sec:tran.cap}, we will take $p$ as a function of $t$ and study the condition on convergence of the SB algorithm.

As in Section \ref{sec.cim}, we introduce the function $U_h(x)$ which is the potential of $H_s$ of as
\begin{align}
	U_{h}(x) = \sum_{i=i}^n \frac{K}{4} x_i^4 +\l(\frac{\Dl-p}{2}\r)x_i^2 -\frac{\xi_0}{2}x^TJx.
\end{align}
Define the project maps $\pi_k: C(H_k)\to \cC(U_h), (x,y) \mapsto x$ and $\pi_s: C(H_s)\to \cC(U_h), (x,y) \mapsto x$.
We first discuss the correspondence between the critical points $\cC(H_k)$ (resp. $\cC(H_s)$) of $H_k$ (resp. $H_s$) and $\cC(U_h)$.

\begin{lemma}\lb{lem:kpo.sb}
	When $p > \lm'_{\max}$ where $\lm'_{\max}$ is the largest eigenvalue of $J$, if $(x,y) \in \cC(H_k)$, then $y = 0$. Furthermore,
	$\pi_k$ is a  well-defined bijection and $i_{H_k}((x,y)) = i_{U_h}(x)$.
\end{lemma}

Since the proof of this lemma is similar as the one of Lemma \ref{lem:DOPO}, we give the sketch of the proof.

\begin{proof}[The sketch proof of Lemma \ref{lem:kpo.sb}]
	When $p > \aa_*^2$, the critical point $(x,y)$ of $\nabla H_k = 0$ satisfies that $y = 0$ and $x = (x_1,\dots, x_n)$ is the root of
	\begin{align}
		-\left(p-\Delta-Kx_{i}^{2}\right) x_{i}+\xi_{0} \sum_{j=1}^{n} J_{i, j} x_{j} =0, \quad 1 \leq i \leq n.
	\end{align}
	Therefore, if $(x,y) \in \cC(H_k)$, then $y=0$ and $x\in \cC(U_h)$ and vice versa. It follows that $\pi_h$ is a bijection.
	
	Note that the Hessian of $H_k$ at the critical point $(x,0)$ is given by $D^2 H_k(x,0) = \diag\{3KX + (\Dl-p)I_n - \xi_0 J, (\Dl+p)I_n-\xi_0 J\}$,
	where $X = \diag\{x_1^2, \dots, x_n^2\}$.
	Note that $(\Dl+p)I_n-\xi_0 J$ is positively definite when $p > \lm_{\max}'$. It follows that $i_{H_k}((x,y)) = i_{U_h}(x)$. This lemma follows.
\end{proof}

\begin{proposition}\lb{prop:mini.KPO.SB}
	When $p>  \max\{\aa_*^2, \lm'_{\max}\}$, if $(x,y)$ is a global minimum point of $H_{k}(x,y)$, then $\sgn (x)$ minimizes the Ising model \eqref{eqn:Ising,kpo}. 	
\end{proposition}

\begin{proof}
	By Lemma \ref{lem:DOPO}, if $(x,0)\in \cC(H_k)$, then $H_k(x,0) = U_h(x)$. Therefore, if $(x,0)$ minimizes $H_k$ globally, then $x$ minimizes $U_h$ globally.
	Via re-scaling of $x$ and Theorem \ref{thm:main}, the signum vector of the global minimum of $U_h$ is a minimizer of the Ising model $E(v) = -\frac{1}{2}v^T J v$. It follows that if $x$ minimizes $H_k$ globally, the signum vector $\sgn(x)$ is the minimizer of Ising model.
\end{proof}

Following the same argument, the similar results of SB algorithm hold.

\begin{proposition}\lb{prop:mini.KPO.SB2}
	When $p>\max \{\aa_*^2, \lm'_{\max}\}$, if $(x,y) \in \cC(H_k)$, then $y = 0$. Moreover, the map $\pi_k$ ($\pi_s$) is a well-defined bijection and $i_{H_k}((x,y)) = i_{U_h}(x)$. If $(x,y)$ is a global minimum point of $H_{s}(x,y)$, then $\sgn (x)$ minimizes the Ising model \eqref{eqn:Ising,kpo}.
\end{proposition}

If $p$ is a constant, it is impossible to achieve the local minimum point of $H_k(x,y)$ or $H_s(x,y)$ along any solution because solutions of these systems preserve the Hamiltonian energy with $\frac{\d H_k}{\d t} = 0$ and $\frac{\d H_s}{\d t} = 0$. Therefore, it is necessary to assume that $\dot p > 0$ when searching for the global minimum of the Hamiltonian function.

When $\dot p > 0$, the Hamiltonian function of KPO may not always decrease along any solution because  $\frac{\d H}{\d t}$ is not always negative with
\begin{align}
	\frac{\d H_k}{\d t} = -\frac{\dot p}{2}\sum_{i=1}^{n}\l(x_i(t)^2-y_i(t)^2\r).
\end{align}
In SB algorithm, the Hamiltonian decreases along any solution because
\begin{align}
	\frac{\d H_s}{\d t} = -\frac{\dot p}{2}\sum_{i=1}^{n}\l(x_i(t)^2+y_i(t)^2\r) < 0.
\end{align}
Therefore, from the dynamical point of view, the SB algorithm shows advantages over the KPO in
achieving the global minimum point of the Hamiltonian function. As shown in Figure 2 of \cite{goto2019combinatorial}, the SB algorithm also preforms better than CIM in some numerical experiments. Therefore, we explore more dynamical properties of the SB algorithm in Section \ref{sec:cap}.

\setcounter{equation}{0}
\setcounter{figure}{0}
\section{Transit and  Capture in SB algorithm}\lb{sec:tran.cap}
We first discuss the transit and capture of SB algorithm when $\dot \aa(t) = 0$ in system \eqref{eqn:SB.11}, then study the dynamics at the saddle in $\R^2$ as an example of the ``neck''. Last we investigate  the capture set of the SB algorithm when  $\dot \aa(t) > 0$ in system \eqref{eqn:SB.11} to illustrate the convergence of SB algorithm.

\subsection{Transit and capture in autonomous Hamiltonian} \lb{sec:transition}

In this section, suppose that $\aa > \aa_*$ is a constant. Then the system \eqref{eqn:SB.11} is autonomous.

Before discussing the transit and capture orbits, we first introduce some concepts from celestial mechanics.
Consider the Hamiltonian $H$ in \eqref{eqn:SB.Hami} with $\aa(t) = \aa$, i.e.,
\begin{align}
	H (x,y) =\sum_{i=1}^{n} \frac{1}{2}y_{i}^{2} + U(x) =\sum_{i=1}^{n} \frac{1}{2}\dot{x}_{i}^{2} + \sum_{i=1}^{n} \l( \frac{1}{4}x_{i}^{4}+\frac{\bb-\aa^2}{2}x_{i}^{2}\r) -\frac{1}{2} x^T Sx. \lb{eqn:tc.H}
\end{align}
The energy surface or the level set $\Sg_c = H^{-1}(c)$ of given Hamiltonian energy $c$ is preserved under the Hamiltonian flow of the vector field \eqref{eqn:SB.11} with $\aa(t) = \aa$.

Define the {\it projection} 	
$\pi: \R^{n} \times \R^{n} \to \R^{n}, (x,y)\mapsto x $.
Following the convention of celestial mechanics, we define the {\it Hill's region} as $\cR_c := \pi(\Sg_c)$ which is the shadow of $\Sg_c$ under the projection.
Since kinetic energy $\sum_{i=1}^{n} \frac{1}{2}y_{i}^{2}$ in the Hamiltonian \eqref{eqn:tc.H} is non-negative, then the Hill's region $\cR_c$ is given by the sub-level set of the potential function $U(x)$ as
\begin{align}
	\cR_c := \{x\in \R^n| U(x)< c\}. \lb{eqn:hill}
\end{align}
Note that $U(0) =0$ and $\lim_{|x| \to \infty} U(x) = \infty$. One can prove that $\cR_c$ for any $c$ is a bounded subset of $\R^n$. The Hill's region in $\R^2$ is shown in Figure \ref{fig:Hill.R2}.

Suppose $x_1$ and $x_2$ are two different local minimum points of the potential $U(x)$. By Corollary \ref{coro:minimum}, $\sgn(x_1) \neq \sgn(x_2)$. We define the {\it neighborhood of minimum point} as follows.

\begin{definition}
	When $\aa > \aa_*$, suppose that $x$ is a local minimum point of $U(x)$ and $N(x)$ is a path-connected neighborhood of $x$. We call $N(x)$ as the neighborhood of minimum point $x$ if $N(x) \cap  \cC(U) =\{x\} \subset \cC_0(U)$, $D^2 U|_{N(x)}$ is positively definite, and $\sgn(x) = \sgn(a)$ for any $a\in N(x)$. We denote the neighborhood of minimum point $x$ as $\Nn(x)$.
\end{definition}

Note that for some neighborhood of $x_1$ and $x_2$ with $x_1,x_2 \in \cC_0(U)$, we must have that $\Nn(x_1) \cap \Nn(x_2) = \emptyset$ because $\sgn(x_1) \neq \sgn(x_2)$.
By the definition of the neighborhood of minimum, the transit and capture orbit $x$ of the SB algorithm are defined as follows.
\begin{definition}\lb{def:cap.tst}
	Suppose $\aa > \aa_*$, $c$ is a constant and $x(t) \in C^2 (\R, \cR_c)$ is one orbit of \eqref{eqn:SB.11}. The orbit $x(t)$ is called  {\it transit} on $I\subset \R$, if there exit some $t_1$ and $t_2$ in $I$, two different local minimum points $x_1$ and $x_2$ and two corresponding $\Nn(x_1)$ and $\Nn(x_2)$ such that $x(t_1) \in \Nn(x_1)$ and $x(t_2) \in \Nn(x_2)$; the orbit $x(t)$ is called capture, if there exists $t_3$ and $x_3 \in \cC_0(U)$ such that $x(t_3) \in \Nn(x_3)$, and when $t\geq t_3$,  $x(t) \notin \Nn(x_4)$ for any $x_4\in \cC_0(U)\bs \{x_3\}$ and any neighborhood of minimum point $x_4$.
\end{definition}

By Definition \ref{def:cap.tst}, we can see that the necessary condition of transit on the energy surface $\cR_c$ is the existence of a continuous path $g_0(t): [0,1] \to \cR_c$ satisfying $g_0(0)\in \Nn(x_1)$ and $g_0(1)\in \Nn(x_2)$ where $x_1$ and $x_2$ are two different local minimum points.
Via Morse theory, we prove there exists a path in $\cR_c$ connecting $g_0(0)$ with $x_1$ in following lemma.

\begin{lemma}\lb{lem:nage.flow}
	There exists a path $g_1\in C^2([0,1],\cR_c)$ satisfying $ g_1(0) = g_0(0)$ and $g_1(1) =x_1$.
\end{lemma}

\begin{proof}
	We take the negative gradient flow of $U(x)$ as
	\begin{align}\lb{eqn:phi.nega}
		\phi: [0,\infty) &\longrightarrow \R^n, \\
		\dot \phi(t)&= -\nabla U(\phi(t)), \nn\\
		\phi(0)& =  g_0(0).\nn
	\end{align}
	Along the solution $\phi(t)$ of \eqref{eqn:phi.nega}, $U(x)$ decreases because
	\begin{align}
		\frac{\d}{\d t} U(\phi(t)) &= \< \nabla U(\phi(t)), \dot \phi(t)\> = - |\dot \phi(t)|^2< 0,
	\end{align}
	where $\< \cdot \>$ is the inner product in $\R^n$ and $|\cdot|$ is the norm in $\R^n$.
	Therefore, for any $t \in \R^+$, we have $U(\phi(t)) <U(g_0(0))$ for any flow $\phi(t)$ of \eqref{eqn:phi.nega}.
	
	Denote $\phi(\infty)= \lim_{t\to \infty} \phi(t)$. Then $U(\phi(\infty)) < U(\phi(0))$. It follows that the flow must be bounded. Namely, $\max_{t\in \R^+}|\phi(t)| < \bar B_1$. It follows that $\max_{t\in [0,\infty)} |U(\phi(t))| < \bar B_2$ for some constant $\bar B_2$. Note that the Hill's region $\cR_{\bar B_2} = \{x|U(x) \leq \bar B_2\}$ is bounded. This yields that both $U(x)$ and $\nabla U(x)$ are bounded in $\cR_{\bar B_2}$. Hence, both $|\dot \phi(t)|$ and $|\ddot \phi(t) |$ are bounded because
	\begin{align}
		\ddot \phi(t) = D^2 U(\phi(x)) \dot \phi(x) = -D^2 U(\phi(x)) \nabla U(x). \lb{eqn:dd.phi}
	\end{align}
	Namely, $|\dot{\phi}|_{C^1([0,\infty),\R^n)} = \sup_{t\in[0,\infty)}|\dot \phi(t)| +  \sup_{t\in[0,\infty)}|\ddot \phi(t)|$ is bounded and $\dot{\phi}$ is uniformly continuous on $[0, \infty)$. Also we have
	\begin{align}
		\int_{0}^{\infty} |\dot \phi(t)|^2 \d t
		=\int_{0}^{\infty} -\frac{\d}{\d t} U(\phi(t)) \d t= U(\phi(0)) -U(\phi(\infty)) < \infty, \nn
	\end{align}
	where the first equality holds by $\frac{\d}{\d t} U(\phi(t)) = \<\nabla U(\phi(t)), \dot \phi(t) \>$.	By uniform continuity of $\dot \phi(t)$, it follows that
	\begin{align}
		\lim_{t\to \infty} \dot \phi(t) =	\lim_{t\to \infty}  \nabla U(\phi(t)) = 0.
	\end{align}
	By the compactness of $\cR_{\bar B_2}$, the Palais–Smale condition (cf. Page 3 of \cite{MR845785}) holds.
	Since $D^2U$ are positively definite in $\Nn(x_1)$ and $x_1$ is the unique critical point in $\Nn(x_1)$, there exists a sequence of $t_n$ satisfying $t_n \to \infty$ such that $\phi(t_n)$ converges to the critical point $x_1$. Namely, $\lim_{n\to \infty} \phi(t_n) = x_1$. By re-scaling $t$ and compactification of the flow, we obtain the path $g_1(t)\in C^2([0,1],\cR_c)$ satisfying $ g_1(0) = g_0(0)$ and $g_1(1) =x_1$.
\end{proof}

By Lemma \ref{lem:nage.flow}, there is  $g_2\in C^2([0,1],\cR_c)$ satisfying $g_2(0) = g_0(1)$ and $g_2(1) =x_2$.
Via concatenation, the path $g(t) = g_1^{-1} *g_0*g_2 (t):[0,1] \to \cR_c$ connects two local minimum points $x_1$ and $x_2$ in $\cR_c$ where $g_1^{-1}$ is the inverse path of $g_1$ in Lemma \ref{lem:nage.flow}. Therefore, transit implies that the existence of a path $g\in C([0,1], \cR_c)$ connecting two local minimum points $x_1$ and $x_2$.

\begin{proof}[Proof of Theorem \ref{thm:mpt}]
	When $\aa> \aa_*$, $U_s > U_M$ holds by Proposition \ref{prop:Us>um}.
	Suppose that $U_{\min} = \min_{x\in \R^n} U(x)$ is the global minimum value of potential $U(x)$.
	
	We prove this theorem by contradiction. Suppose that $x(t)$ is a transit orbit in $\cR_{U_s}=\{x|U(x) < U_s\}$ with $x(t_1)\in \Nn(x_1)$ and $x(t_2) \in \Nn(x_2)$. By Lemma \ref{lem:nage.flow}, there exists a continuous path $g_*(t) \subset \cR_{U_s}$ with $g_*(0) = x_1$ and $g_*(1) = x_2$. Define $c_0$ as
	\begin{align}
		c_0 = \inf_{g\in \Lm} \max_{t\in [0,1]} U(g(t)),
	\end{align}
	where $\Lm=\{g(t) \in C^2([0,1], \cR_{U_s})|g(0) = x_1, g(1) = x_2\}$. Note that $g_* \in \Lm$.
	By the definition of $c_0$, $c_0 \leq \max_{t\in[0,1]} U(g_*(t)) < U_s$ and $\{x|U(x)<\max_{t\in[0,1]} U(g_*(t)) \}$ is compact. It follows that the Palais–Smale condition holds.
	By the deformation lemma (cf. Theorem A.4 of \cite{MR845785}) and the moutain pass theorem (cf. Theorem 2.2 of \cite{MR845785}), there exists at least one mountain pass point $x_0 \in\cR_{U_s}$ such that $c_0 = U(x_0) < U_s$ and $\nabla U (x_0) = 0$.
	According to Hofer in \cite{MR812787} or Tian in \cite{MR794263}, the Morse index of $x_0$ satisfies $i_{U}(x_0) = 1$. Therefore,
	$x_0$ is a saddle. It contradicts $U(x_0)< U_s$. Hence, the transit is impossible.
	
	If the orbit is not capture, then $x(t)$ is a transit orbit. But when $c < U_s$, the transit is impossible. Then this theorem follows.	
\end{proof}

\subsection{Transit in $\R^2$}\lb{sec:transit.R2}
When $\aa^2> \bb-2$, the values of critical points of $U(x)$ can be classified into
\begin{align}
	c_0 & = U(0) =0, \lb{eqn:r2.c0}\\
	c_1 &= U(x_1)= -\frac{(\aa^2-\bb)^2}{4}+\frac{1}{2},\lb{eqn:r2.sad}\\
	c_2 &= U(x_2)= -\frac{(\aa^2-\bb)^2}{2}-(\aa^2-\bb)+\frac{3}{4},\lb{eqn:r2.min1}\\
	c_3 &= U(x_3)= -\frac{(\aa^2-\bb)^2}{2}-(\aa^2-\bb) -\frac{1}{2},\lb{eqn:r2.min2}
\end{align}
where $x_1\in \cC_s(U)$, $x_2 \in \{(\lm_2,-\lm_2), (-\lm_2,\lm_2)\}$ is a local minimum point and $x_3 \in \{(\lm_1,\lm_1), (-\lm_1,-\lm_1)\}$ is a global minimum point as in Table \ref{tab:1}.
Applying Theorem \ref{thm:mpt} directly, the following proposition holds.
\begin{proposition}
	If the orbit $x(t)$ is transit in $\R^2$, then $H(\dot{x},x) > c_1$; if $H(\dot{x},x) < c_1$, then $x(t)$ is a capture orbit in $\R^2$.
\end{proposition}

When Hamiltonian energy is slightly bigger than $c_1$, the saddles of the potential look like a ``neck''. When the Hamiltonian energy is equal to $c_1$ or slightly bigger $c_1$, the dynamics near the ``necks'' is observed as the ones in \cite{MR233535}.

In the rest of this section, we follow the convention in celestial mechanics by changing the order of momentum and position to $(y,x)$ in order to simplify computations. Take the saddle $z_0 = (0,0,\lm_3, -\lm_4)$ as an example of the "neck". Abusing the notations, we still write the solution of the linearized Hamiltonian system as $\ga(t) = (y_1,y_1,x_1,x_2)$.
The linearized Hamiltonian system at $z_0$ is given by
\begin{align}
	\dot{\ga} = J_4 D^2H(z_0)\ga, \lb{eqn:line.hami}
\end{align}
where $J_4 = (\begin{smallmatrix}
	0 &-I_2\\
	I_2 & 0
\end{smallmatrix})$ is the standard symplectic matrix and $D^2H(z_0)$ is given by
\begin{align}
	D^2H(z_0)=
	\begin{pmatrix}
		1 & 0 & 0 & 0\\
		0 & 1 & 0 & 0\\
		0 & 0 & 3\lm_3^2+\bb-\aa^2 & -1\\
		0 & 0  & -1 & 3\lm_4^2+\bb -\aa^2
	\end{pmatrix}.
\end{align}
The characteristic polynomial of $J_4 D^2H(z_0)$ is $\mu^4 +( 3(\lm_3^2+\lm_4^2)+2\bb-2
\aa^2)\mu^2 + (3\lm_3^2+\bb-\aa^2)(3\lm_4^2+\bb -\aa^2)-1$.
Its eigenvalues are given by $\sg(J_4 D^2H(z_0)) = \{-\mu_1, \mu_1,-\mu_2,\mu_2\}$ where
\begin{align}
	\mu_1 &= \frac{\sqrt{\sqrt{(9(\aa^2-\bb)^2-32}-(\aa^2-\bb)}}{\sqrt{2}} > 0,\\
	\mu_2 &= \frac{\sqrt{-\sqrt{(9(\aa^2-\bb)^2-32}-(\aa^2-\bb)}}{\sqrt{2}} \in \ii\R.
\end{align}
Note that $\aa^2-\bb > 2$, the $\mu_i$s are all well-defined. The corresponding eigenvectors are
$e_1=(-\mu_1 u,-\mu_1,u,1)$, $e_2 =(\mu_1u, \mu_1,u,1)$, $e_3 =(\mu_2v,-\mu_2,-v,1)$, and $e_4 =(-\mu_2v, \mu_2,-v,1)$
where $u = \frac{1}{2} \sqrt{9(\aa^2-\bb)^2-32}-\frac{3}{2}\sqrt{(\aa^2-\bb)^2-4} >0$ and $v=\frac{1}{2} \sqrt{9(\aa^2-\bb)^2-32}+\frac{3}{2}\sqrt{(\aa^2-\bb)^2-4}>1$.
Therefore, the general real solution of \eqref{eqn:line.hami} is given by
\begin{align}
	\ga(t) = \xi_1 e_1 \exp(-\mu_1 t) + \xi_2 e_2 \exp(\mu_1 t) + 2 \RRe (\eta e_3 \exp(\mu_2 t)), \lb{eqn:dd.soln}
\end{align}
where $\xi_1$, $\xi_2$ are real numbers and $\eta$ is a complex number.

By projecting map, the general real solutions in the $x_1x_2$-plane fall into nine different classes by the limit behavior of $x=(x_1(t), x_2(t))$.
Since $(\lm_3, -\lm_4)$ is between $(\lm_1,\lm_1)$ and $(\lm_2, -\lm_2)$, we consider the behavior of $x_2(t)$.
If $t\to -\infty$, $x_2$ can tend to negative infinity, be bounded, or tend to positive infinity according to the sign of $\xi_1> 0$, $\xi_1 = 0$ or $\xi_1 < 0$ respectively. The same statement holds for $t\to \infty$ and $\xi_2$ replace $\xi_1$. If $x_2$ is bounded (in either direction), then the corresponding limit set is unique (up to time translation). The periodic solutions are determined by $\xi_1=\xi_2 =0$.
\begin{proposition}
	If $\xi_1=\xi_2=0$, the periodic orbit  projects into $x$-plane as an ellipse with the major axis of the length $2v|\eta|$ and the minor axis of the length $2|\eta|$. Furthermore, its motion is clockwise.
\end{proposition}
\begin{proof}
	From \eqref{eqn:dd.soln}, project $\ga(t)$ to the $x$-plane and obtain that
	\begin{align}
		x_1(t) = -2 v \RRe(\eta \exp(\mu_2 t)), \quad x_2(t) = 2 \IIm(\eta \exp(\mu_2 t)), \nn
	\end{align}
	where $v > 0$. Therefore, the motion is clockwise.
\end{proof}

The dynamics near the ``neck'' is not simply transit. There are also asymptotic orbits and capture orbits.
If $\xi_1\xi_2 = 0$, the orbits are asymptotic to the periodic solution in the equilibrium region; if $\xi_1\xi_2<0$, the orbits ``cross'' the equilibrium region of the saddle point form $-\infty$ to $\infty$ or inversely; and if $\xi_1\xi_2>0$, the orbits are captured namely it cannot cross the equilibrium region.

We believe that the phenomena of the ``neck'' also exists when $n\geq 3$. However, it will be much more involved.

\subsection{Capture of SB algorithm}\lb{sec:cap}
In this section, we assume that $\aa$ is a function of $t$  with $\dot{\aa}(t) > 0$ and $\lim_{t \to \infty} \aa(t) = \aa_{\infty}$ where $\aa_{\infty} > 8\aa_*$ is a sufficiently large constant and $\aa_*$ is given in Theorem \ref{thm:main}. It follows that there exists a $t_0$ such that $\aa(t_0)> \aa_*$. The results in Section \ref{sec:opt.set} hold for all $t> t_0$.
Note that $U_s(t)$ in \eqref{eqn:Us} changes along the time.
Since $\aa(t) > 0$ and $\dot{\aa} > 0$, the Hamiltonian decreases with $t$ along any solution. Namely, $\frac{\d H}{\d t} = -2\aa \dot{\aa}\sum_{i = 1}^n x_i^2<0$.

For simplicity, we still use $\nabla$ to denote the gradient with respect to $x$. Since $U_s(t)$ is relevant in capture by Theorem \ref{thm:mpt}.
We first estimate the value of $U_s(t)$.

\begin{lemma}\lb{lem:U_s}
	There exist a positive constant $B_5$ and $t_1\geq t_0$ such that for $t\geq t_1$,
	\begin{align}\lb{eqn:b5.def}
		\l|\frac{U_s(t)}{\aa^2(t)} +\frac{(n-1)\aa^2(t)}{4} \r| <  B_5.
	\end{align}
\end{lemma}
\begin{proof}
	For any given $t> t_0$, there is an $x_0 \in \cC_s(U)$ with $i_{U}(x_0) = 1$ satisfying $U(x_0,t) =U_{s}(t)$. We can write $x_0 = \bar{x}_0+\dl_0$ where $\bar{x}_0 = (\bar x_1,\dots,\bar x_n)^T$ and $\dl_0=(\dl_1,\dots,\dl_n)^T$.
	We omit $t$ in $\aa(t)$, $x_0(t)$, $U_s(t)$ and $U(x_0(t),t)$ in following formula to simplify the notations.
	\begin{align}
		U_s= U(\bar{x}_0+\dl_0)
		=& \sum_{i = 1}^{n} \frac{1}{4} (\bar{x}_i+\dl_i)^4 + \sum_{i = 1}^{n} \frac{\bb-\aa^2}{2}(\bar{x}_i+\dl_i)^2 - \frac{1}{2}(\bar{x}+\dl)^T S(\bar{x}+\dl) \nn \\
		=&-\frac{(n-1)\aa^4}{4} + \frac{\bb(n-1)\aa^2}{2} \nn	\\
		&\quad +\sum_{i = 1}^{n}\l(\bb\bar{x}_i \dl_i+\frac{3\bar{x}^2_i\dl^2_i}{2}   +\frac{(\bb-\aa^2) \dl_i^2}{2}+\bar{x}_i\dl_i^3+\frac{\dl_i^4}{4}\r) \nn \\
		&\quad - \frac{1}{2}\bar{x}^T S\bar{x} - \bar{x}^TS\dl - \frac{1}{2}\dl^T S\dl.\nn
	\end{align}
	When $\aa$ is sufficiently large, both $ \frac{\bb(n-1)\aa^2}{2}$ and $- \frac{1}{2}\bar{x}^T S\bar{x} $ possess the order $\aa^2$, while $\bb\bar{x}_i \dl_i$,$\frac{3\bar{x}^2_i\dl^2_i}{2}$,  $\frac{(\bb-\aa^2) \dl_i^2}{2}$, $\bar{x}_i\dl_i^3$, $\frac{\dl_i^4}{4}$, $\bar{x}^TS\dl$, and $ \frac{1}{2}\dl^T S\dl$ are bounded. Then there exist a positive $B_5 > 0$ and $t_1 > t_0$ such that for any $t > t_1$,
	\begin{align}
		\l|\frac{U_s(x)}{\aa^2} +\frac{(n-1)\aa^2}{4} \r| <  B_5.
	\end{align}
	This yields this lemma holds.
\end{proof}

We define $U_B(t)$ as
\begin{align}
	U_B(t) := -\frac{(n-1)\aa^4(t)}{4} - B_5 \aa^2(t). \lb{eqn:U_B}
\end{align}
It follows that for all $t > t_1$, $U_s(t) > U_B(t)$ by Lemma \ref{lem:U_s}.
According to Corollary \ref{coro:minimum}, when $t\geq t_0$, there exists $B_6$ such that every $x\in \cC_0(U)$ satisfies
\begin{align}
	|x|^2 \geq n \aa^2(t) + B_6.
\end{align}
Since $\aa_{\infty}$ is sufficiently large, we have
\begin{align}
	\aa^2_{\infty} \gg 2(B_5-B_6). \lb{eqn:ass.aa.infty}
\end{align}
There exists $t_2$ such that
\begin{align}
	\aa(t)^2 > \frac{1}{2} \aa^2_{\infty}, \lb{eqn:at.ainfty}
\end{align}
when $t > t_2$.
Define the constant $R_0$ and the function $U_{R_0}(t)$ as
\begin{align}
	R_0 &:= \frac{(n-1)\aa^2_{\infty}}{2} + B_5, \lb{eqn:R0}\\
	U_{R_0}(t) &:= \min_{|x|^2 =R_0} U(x,t). \lb{eqn:UR0}
\end{align}

Together with \eqref{eqn:ass.aa.infty} and \eqref{eqn:at.ainfty}, there exists $t_3 >  \max\{t_0, t_1, t_2\}$ such that
\begin{align}
	\aa^2(t) \geq \frac{(n-1)\aa^2_{\infty}}{2n} + \frac{(B_5 - B_6)}{n},  \lb{eqn:p.geq.t0}
\end{align}
when $t> t_3$.
For any given $t$ with $t\geq t_3$, the norm of every local minimum point of $U(x,t)$ satisfies
\begin{align}
	|x|^2 \geq R_0. \lb{eqn:x.footprint R0}
\end{align}
At the global minimum point $x_{\min}$ of $U(x)$, we have that $U(x_{\min}(t),t) < U_{R_0}(t)$ when $t > t_3$.
By the definition of $\P(t)$ in \eqref{eqn:p.t}, we have following lemma holds.

\begin{lemma}\lb{lem:cap.nonempty}
	The set $\P(t)$ is non-empty for $t > t_3$.
\end{lemma}
\begin{proof}
	For any given $t \geq t_3$, the global minimum point of $U(x)$ satisfies $U_{\min}(t) < \min\{U_{R_0}(t), U_{s}(t)\}$.  By the continuity of $U(x)$ in $x$, there exists a neighborhood $\Nn(x_{\min})$ of $x_{\min}$  such that $x\in \Nn(x_{\min})$, $U_{\min}(t) \leq U(x,t) <  \min\{U_{R_0}(t), U_{s}(t)\}$. Then the set $\P(t)$ is non-empty for any $t> t_3$.
\end{proof}

Now we are ready to prove that $\P(t)$ in \eqref{eqn:p.t} is a capture set.

\begin{proof}[Proof of Theorem \ref{thm:cap}]
	The first step is to prove that if there exists some $t_* \geq t_3$ such that $|x(t_*)|^2 >  R_0$ and $H(x(t_*),t_*) < U_{R_0}(t_*)$, then for all $t\geq t_*$, $H(x(t),t) < U_{R_0}(t)$ and $|x(t)|^2 >  R_0$ hold.
	Note that if there exists $t'$ such that $|x(t')|^2 = R_0$ then $H(x,t') \geq U(x,t') \geq U_{R_0}(t)$. Hence, we only need to prove that for all $t\geq t_*$, $H(x(t),t) < U_{R_0}(t)$ holds.
	
	We prove this by contradiction. We assume there exists $t_4$ such that $H(x(t_4),t_4) >U_{R_0}(t_4)$.
	By the continuity of $H(x(t),t)$ and $U_{R_0}(t)$, we can find
	\begin{align}
		\bar t = \inf \{t> t_*|H(x(t),t) = U_{R_0}(t)\}.
	\end{align}
	Then $H(x(t),t) < U_{R_0}(t)$ and $|x|^2 \geq R_0$ for all $t\in (t_*,\bar t)$.
	Then we have that
	\begin{align}
		H(x(\bar t),\bar t) = H(x(t_*),t_*) + \int_{t_*}^{\bar t} \frac{\d H}{\d t}\d t=  H(x(t_*),t_*) -  \int_{t_*}^{\bar t}\aa \dot{\aa}\sum_{i = 1}^n x_i^2(t) \d t,
	\end{align}
	and for any $x$ satsifies $|x|^2=R_0$,
	\begin{align}
		U(x, \bar t)|_{|x|^2=R_0}
		=U(x,t_*) -  \int_{t_*}^{\bar t}  \aa \dot{\aa}\sum_{i = 1}^n x_i^2 \d t
		=U(x,t_*) -  \int_{t_*}^{\bar t}  \aa \dot{\aa} R_0 \d t,
	\end{align}
	where the last equation holds by $|x|^2 =R_0$.
	Note that $H(x(t_*),t_*) < U(x, t_*)$ and $|x(t)|> R_0$ for all $t\in [t_*,\bar  t)$. Therefore, $H(x(\bar t),\bar t) < U_{R_0}(\bar t)$. It contradicts the assumption.
	Therefore, if $ H(x(t_*),t_*) < U_{R_0}(t_*)$ and $|x(t_*)|^2 > R_0$, then $H(x(t),t) < U_{R_0}(t)$ and $|x(t)|^2 >R_0$ for $t\geq t_*$.
	
	The rest is devoted to proving $H(x(t),t) < U_{B}(t)$ for all $t\geq t_*$.
	The derivative $U_{B}$ is given by
	\begin{align}
		\frac{\d U_{B}}{\d t} &= -\dot{\aa} \l((n-1)\aa^3+B_5\aa \r).
	\end{align}
	Since $x(t)\in \P(t)$ and \eqref{eqn:R0}, we have that $|x|^2 > R_0^2=\frac{(n-1)\aa^2_{\infty}}{2}+B_5 > \frac{(n-1)\aa^2}{2}+B_5$ for all $t\geq t_3$. It yields that for all $t\geq t_*$,
	\begin{align}
		\frac{\d H(x,t)}{\d t} = -\dot{\aa}\aa\sum_{i = 1} ^n x_i^2\leq -\dot{\aa}\aa \l( \frac{(n-1)\aa^2}{2}+B_5 \r) = \frac{\d U_{B}}{\d t}.\lb{eqn:dH.le.dU}
	\end{align}
	Since $U_{B}(t) = U_{B}(t_*) + \int_{t_*}^t \frac{\d U_{B}(t)}{\d t} \d t$, then along the orbit $x(t)$,
	\begin{align}
		H(x(t),t) <  U_{B}(t), \nn
	\end{align}
	for all $t\geq t_*$. Then Theorem \ref{thm:cap} holds.	
\end{proof}

\subsection{Capture in $\R^2$}\lb{sec:cap.r2}
When $n=2$, the two axes divide $\R^2$ into four connected components, i.e., the four quadrants.
If $x(t)$ is captured by one quadrant, then $\sgn(x(t)) \in C(E)$.

Instead of considering $U_B$ in \eqref{eqn:U_B}, we restrict $U(x)$ in \eqref{eqn:pote} to one axis directly because the topology of $\cR_c$ is much simpler than $\R^n$. Take $x_2 = 0$ as an example. It follows that
\begin{align}
	U|_{x_2= 0} = \frac{1}{4} x_{1}^{4}+\frac{\bb-\aa^2(t)}{2} x_{1}^{2}. \nn
\end{align}
Then $U|_{x_2= 0}$ possesses three critical points where
$x_1 = 0$ is the local maxima and $x_1 = \pm \sqrt{\aa^2(t)-\bb}$ are two local minimum points. Suppose that $x_{sd}(t) := (\sqrt{\aa^2(t)-\bb},0)$ and
\begin{align}
	U_{sd}(t) = U(x_{sd},t) = -\frac{(\aa^2(t)-\bb)^2}{4}. \nn
\end{align}
As the proof in Theorem \ref{thm:cap}, if $H(t)<U_{sd}(t)$, the transit is impossible.
We define that $t_0$ by
$ \aa(t_0)^2 = \frac{3}{4}\aa_{\infty}^2+\frac{1}{4}\bb$,
then
$
\aa^2(t_0)> \frac{1}{2}(\aa^2_{\infty}+\bb) -1.
$
Via direct computations, one can verify the existence of $t_0$.

Suppose that $R_0 := \frac{\aa^2_{\infty}-\bb}{2} > 0$. Note that $R_0 > \frac{\aa^2(t)-\bb}{2}$ for all $t >0$.
The restriction of $U(x,t)$ on the $\{x=(x_1,x_2)|x_1^2+x_2^2 = R_0\}$ is given by
$
U_{R_0}(t) = U(x,t)|_{|x|^2 = R_0}.
$
The capture set of orbit $\P_2(t)$ is defined as
\begin{align}
	\P_2(t) = \{x = (x_1,x_2)| H(x(t),t) \leq \min\{U_{R_0}(t), U_{sd}(t)\}\}. \nn
\end{align}

When $|x| < R_0$ and $t > t_0$, $U(x)$ is a super-harmonic function by
\begin{align}
	\Delta U &=\frac{\pt^2 U}{\pt x_1^2} +  \frac{\pt^2 U}{\pt x_2^2}< 3R_0  + 2(\bb - \aa^2(t)) < 0, \nn
\end{align}
where $\aa^2(t) \geq \frac{3}{4}\aa_{\infty}^2+\frac{1}{4}\bb$ for all $t > t_0$. By the weak minimum principle (cf. Theorem 2.3 of \cite{DavidGilbarg2001}), we have that
$ \inf_{|x|< R_0} U = \inf_{|x| = R_0} U$.
Therefore, the inequality $H(x(t),t) \leq U_{R_0}(t)$ yields that $|x(t)| \geq R_0$ if $x(t) \in \P(t)$.
Hence, we can omit the condition on $|x|$ in the definition of $\P_2$.

For any $t\geq t_0$, the potential at the point $x_{\min} = (\lm_1, \lm_1)$ or $x_{\min} = (-\lm_1, -\lm_1)$ achieves its minimum value. Namely,
\begin{align}
	U(x_{\min}) = -\frac{(\aa^2(t)-\bb+1)^2}{2}. \nn
\end{align}
Then $U(x_{\min}) <U_{sd}(t)$. Also $|x_{\min}|^2 = 2\lm_1^2 = 2(\aa^2(t)-\bb + 1)>R_0$. This yields that the set $\P_2(t)$ is non-empty for all $t \geq t_0$.
By Theorem \ref{thm:cap}, we have following proposition.
\begin{proposition}\lb{prop:r2}
	If $x(t_1) \in \P_2(t_1)$ for some $t_1\geq t_*$, then $x(t) \in \P_2(t)$ for all $t\geq t_1$.
\end{proposition}
\begin{remark}
	According to \eqref{eqn:r2.min1}-\eqref{eqn:r2.min2}, we can see that $|c_2-c_3|$ is very small. Also, both $c_2(t)<\min\{U_{R_0}(t), U_{sd}(t)\}$ and  $c_3(t)<\min\{U_{R_0}(t), U_{sd}(t)\}$ in \eqref{eqn:r2.min1}-\eqref{eqn:r2.min2} hold. The capture can happen in the neighborhoods of all the local minimum in $\R^2$. This explains why  the SB algorithm can  convergence to either the local minimum points or the global minimum points as the numerical results in \cite{goto2019combinatorial}.
\end{remark}

\setcounter{equation}{0}
\setcounter{figure}{0}
\appendix
\section*{Appendix}
\section{Some computations}
\begin{lemma}\lb{app:lem:ted.comp}
	For any given $\ep > 0$, $f_{\pm,\ep}(x) =\frac{1}{\aa^2}x^3-x \pm \ep = 0$ possesses three solutions $x_{\pm,i}$ with $1 \leq i \leq 3$ satisfying
	\begin{align}
		\lim_{\aa \to \infty} |x_{\pm,1}+\aa| =\lim_{\aa \to \infty} |x_{\pm,3}-\aa| =  \frac{\ep}{2}, \quad
		\lim_{\aa \to \infty} |x_{\pm,2}| = \ep.
	\end{align}
\end{lemma}
\begin{proof}
	The roots of $f_{\pm,\ep}(x) = 0$ are given by
	\begin{align}
		x_{\pm,1} =& \frac{(-1+\sqrt{-3})\aa^2}{12}\sqrt[3]{-A_{\pm}+\sqrt{A_{\pm}^2 +C^3}}  + \frac{(-1-\sqrt{-3})\aa^2}{12}\sqrt[3]{-A_{\pm}-\sqrt{A_{\pm}^2 +C^3}}, \lb{eqn.f.ep.root1}\\
		x_{\pm,2} =& \frac{\aa^2}{6} \l(\sqrt[3]{-A_{\pm}+\sqrt{A_{\pm}^2 +C^3}} +  \sqrt[3]{-A_{\pm}-\sqrt{A_{\pm}^2 +C^3}}\r),\lb{eqn.f.ep.root2}\\	
		x_{\pm,3} =& \frac{(-1-\sqrt{-3})\aa^2}{12}\sqrt[3]{-A_{\pm} +\sqrt{A_{\pm}^2 +C^3}}  + \frac{(-1+\sqrt{-3})\aa^2}{12} \sqrt[3]{-A_{\pm} -\sqrt{A_{\pm}^2 +C^3}}, \lb{eqn.f.ep.root3}
	\end{align}
	where $A_{\pm} = \pm 108\ep /\aa^4$ and $C = -12/\aa^2$.
	
	Note that $x_{\pm,2}$ can be calculated as
	\begin{align}
		x_{\pm,2} =& \frac{\aa^2}{6} \l(\sqrt[3]{-A_{\pm} +\sqrt{A_{\pm}^2 +C^3}} +  \sqrt[3]{-A_{\pm} -\sqrt{A_{\pm}^2 +C^3}}\r) \nn\\
		=&  \frac{\aa^2}{6} \frac{-2A_{\pm}}{\sqrt[3]{(-A_{\pm}+\sqrt{A_{\pm}^2 +C^3})^2} + \sqrt[3]{(-A_{\pm}-\sqrt{A_{\pm}^2 +C^3})^2}+C} \nn \\
		=& \frac{\aa^2
		}{6} \frac{-2A_{\pm}}{C+\sqrt[3]{2A_{\pm}^2 + C^3-2A_{\pm}\sqrt{A_{\pm}^2 +C^3}} + \sqrt[3]{2A_{\pm}^2 + C^3+2A_{\pm}\sqrt{A_{\pm}^2 +C^3}}} \nn  \\
		=& \frac{\mp 36\ep }{-12+L_1 + L_2 }, \nn
	\end{align}
	where
	\begin{align}
		L_1 =&\aa^2\sqrt[3]{2A_{\pm}^2 + C^3-2A_{\pm}\sqrt{A_{\pm}^2 +C^3}}   \\
		=& \sqrt[3]{2\cdot 108^2 \ep^2 /\aa^2 -12^3   \mp 2\cdot 108 \ep \sqrt{108^2\ep^2/\aa^4 - 12^3/\aa^2}}, \nn \\
		L_2 =&\aa^2 \sqrt[3]{2A_{\pm}^2 + C^3+2A_{\pm}\sqrt{A_{\pm}^2 +C^3}} \\
		=& \sqrt[3]{2\cdot 108^2 \ep^2/\aa^2 -12^3 \pm 2\cdot 108 \ep \sqrt{108^2\ep^2/\aa^2 - 12^3/\aa^2}}. \nn
	\end{align}
	Note that $\lim_{\aa\to \infty} L_1 = \lim_{\aa\to \infty} L_2 = -12$. Then
	\begin{align}
		\lim_{\aa \to \infty} |x_{\pm,2}| = \ep.
	\end{align}
	For $x_{\pm,1}$, we have that
	\begin{align}\lb{eqn:x+-1}
		x_{\pm,1} =& \frac{(-1+\sqrt{-3})\aa^2}{12}\sqrt[3]{-A_{\pm}+\sqrt{A_{\pm}^2 +C^3}} \\
		&\quad + \frac{(-1-\sqrt{-3})\aa^2}{12}\sqrt[3]{-A_{\pm}-\sqrt{A_{\pm}^2 +C^3}}  \nn \\
		=&\frac{-\aa^2}{12}\l(\sqrt[3]{-A_{\pm}+\sqrt{A_{\pm}^2 +C^3}}+\sqrt[3]{-A_{\pm}-\sqrt{A_{\pm}^2 +C^3}}\r) \nn \\
		&+\frac{\sqrt{-3}\aa^2}{12}\l(\sqrt[3]{-A_{\pm}+\sqrt{A_{\pm}^2 +C^3}}-\sqrt[3]{-A_{\pm}-\sqrt{A_{\pm}^2 +C^3}}\r) \nn \\
		=&	\frac{-x_{\pm,2}}{2} + \frac{\sqrt{-3}\aa^2}{6} \frac{\sqrt{A_{\pm}^2 +C^3}}{-C + \sqrt[3]{(-A_{\pm}+\sqrt{A_{\pm}^2 +C^3})^2}+\sqrt[3]{(-A_{\pm}-\sqrt{A_{\pm}^2 +C^3})^2}} \nn \\
		=& \frac{-x_{\pm,2}}{2} + \frac{\sqrt{3\cdot 12^3\aa^2-3\cdot 108^2\ep^2 }}{72 + 6L_1+6L_2}. \nn
	\end{align}
	Therefore,  we have that
	\begin{align}
		\lim_{\aa \to \infty} |x_{\pm,1}+\aa| =\frac{\ep}{2}. \lb{eqn:lim.x.pm.1}
	\end{align}
	Following a similar argument, $x_{\pm,3}$ can be simplified as
	\begin{align}
		\lb{eqn:x+-3}
		x_{\pm,3}
		=& \frac{-x_{\pm,2}}{2} - \frac{\sqrt{3\cdot 12^3\aa^2-3\cdot 108^2\ep^2 }}{72K + 6L_1+6L_2},
	\end{align}
	and
	\begin{align}
		\lim_{p\to \infty} |x_{\pm,3}-\aa| =\frac{\ep}{2}. \lb{eqn:lim.x.pm.3}
	\end{align}
\end{proof}

\begin{lemma}\lb{app:lem:ted.comp2}
	For any given $\ep > 0$, $\td f_{\pm,\ep}(x) =\frac{1}{\aa^2}x^3-x \pm \frac{\ep}{\aa} = 0$ possesses three solutions $\td x_{\pm,i}$ for $1 \leq i \leq 3$. There exist $B_2$ and $\aa_5$ such that for all $\aa > \aa_5$,  $\td x_{\pm,i}$  satisfy
	\begin{align}
		|\td x_{\pm,1}+\aa| < \frac{B_2}{\aa},\quad |\td x_{\pm,2}| < \frac{B_2}{\aa},\quad |\td x_{\pm,3}-\aa| < \frac{B_2}{\aa}.\lb{app.eqn:1.of.3}
	\end{align}
\end{lemma}
\begin{proof}
	The roots of $\td f_{\pm,\ep}(x) = 0$ are given by $\td x_{\pm,1}$, $\td x_{\pm,1}$ and $\td x_{\pm,3}$ as \eqref{eqn.f.ep.root1}, \eqref{eqn.f.ep.root2}, \eqref{eqn.f.ep.root3} by replacing  $A_{\pm}$ with $\td A_{\pm} = \pm 108\ep /\aa^5$. Then $x_{\pm,2}$ can be calculated as
	\begin{align}
		\td x_{\pm,2} =& \frac{\mp 36\ep }{-12 \aa+ \td L_1 + \td L_2 }, \nn
	\end{align}
	where
	\begin{align}
		\td L_1 =&\aa^3
		\sqrt[3]{2\td A_{\pm}^2 + C^3-2\td A_{\pm}\sqrt{\td A_{\pm}^2 +C^3}} \nn\\
		=& \sqrt[3]{2\cdot 108^2 \ep^2 /\aa -12^3 \aa^3   \mp 2\cdot 108 \ep \sqrt{108^2\ep^2/\aa^2 - 12^3\aa^4}}, \nn \\
		\td L_2 =&\aa^3 \sqrt[3]{2 \td A_{\pm}^2 + C^3+2 \td A_{\pm}\sqrt{\td A_{\pm}^2 +C^3}} \nn \\
		=& \sqrt[3]{2\cdot 108^2 \ep^2/\aa -12^3 \aa^3 \pm 2\cdot 108 \ep \sqrt{108^2\ep^2/\aa^2 - 12^3 \aa^4}}. \nn
	\end{align}
	Note that $\lim_{\aa\to \infty} \td L_1/\aa = \lim_{\aa\to \infty} \td L_2/\aa = -12$. Then
	\begin{align}
		\lim_{\aa \to \infty} | \aa \td x_{\pm,2} | = \frac{3}{2}\ep. \nn
	\end{align}
	It yields that for any $\ep$, there exists a $\bar B_1$ such that for $\aa > \bar \aa_1$
	\begin{align}
		|\td x_{\pm,2}| < \bar B_1/\aa . \lb{eqn:lim.x.pm.2}
	\end{align}
	Following a similar argument of \eqref{eqn:x+-1} and \eqref{eqn:x+-3}, there exists $\bar B_2$ such that for $\aa >\bar \aa_2$,
	\begin{align}
		|\td x_{\pm,1}+\aa| < \frac{\bar B_2}{\aa},\quad |\td x_{\pm,3}-\aa| < \frac{\bar B_2}{\aa}.\lb{eqn:1.of.3.2}
	\end{align}
	Then uniform $B_2$ and $\aa_5$ can be founded in \eqref{eqn:lim.x.pm.2} and \eqref{eqn:1.of.3.2} and this lemma holds.
\end{proof}

\section*{Acknowledgments}\lb{sec:appen}
The authors are grateful to Prof. Yiming Long, Prof. Zhenli Xu and Prof. Jie Sun for discussions and encouragement on this topic.
B. Liu also thanks Dr. Liwei Yu and Dr. Hao Zhang for their  discussions on related topics.

\noeqref{eqn:r2.lm2, eqn:r2.lm3, eqn:r2.c0, eqn:r2.sad}

\def\cprime{$'$}

\end{document}